\documentclass[11pt,a4paper]{amsart}

\usepackage{amssymb,amsmath,graphics,verbatim}
\usepackage{latexsym}
\usepackage{eucal}
\usepackage{a4wide}

\usepackage{color}

\newtheorem{theorem}{Theorem}
\newtheorem{lemma}[theorem]{Lemma}

\newtheorem{cor}[theorem]{Corollary}
\theoremstyle{remark}

\newcommand{\mc}{\mathcal}
\newcommand{\rr}{\mathbb{R}}
\newcommand{\nn}{\mathbb{N}}
\newcommand{\cc}{\mathbb{C}}
\newcommand{\hh}{\mathbb{H}}
\newcommand{\zz}{\mathbb{Z}}

\newcommand{\la}{s}
\newcommand{\eps}{\epsilon}

\newcommand{\pl}{\partial}
\newcommand{\x}{\times}

\newcommand{\til}{\widetilde}
\newcommand{\bbar}{\overline}

\newcommand{\supp}{\textrm{supp}}
\newcommand{\cjd}{\rangle}
\newcommand{\cjg}{\langle}

\newcommand{\demi}{\frac{1}{2}}
\newcommand{\ndemi}{\frac{n}{2}}

\newcommand{\ima}{\mathrm{Im}}
\newcommand{\rea}{\mathrm{Re}}
\newcommand{\Op}{{\rm Op}}

\begin{document}
\title[Equidistribution and Eisenstein series]{Equidistribution of Eisenstein series on convex co-compact hyperbolic manifolds}
\date{\today}
\author{Colin Guillarmou}
\address{DMA, U.M.R. 8553 CNRS\\
Ecole Normale Sup\'erieure\\
45 rue d'Ulm\\ 
F 75230 Paris cedex 05 \\France}
\email{cguillar@dma.ens.fr}
\author{Fr\'ed\'eric Naud}
\address{Laboratoire d'Analyse non lin\'eaire et G\'eom\'etrie\\
Universit\'e d'Avignon \\
33 rue Louis Pasteur 84000 Avignon}
\email{frederic.naud@univ-avignon.fr}

\begin{abstract}
For convex co-compact 
hyperbolic manifolds $\Gamma\backslash \hh^{n+1}$ for which 
the dimension of the limit set satisfies $\delta_\Gamma< n/2$, we show that the high-frequency Eisenstein series 
associated to a point $\xi$ ``at infinity'' concentrate microlocally on a measure supported 
by (the closure of) the set of points in the unit cotangent bundle corresponding to geodesics ending at $\xi$.
The average in $\xi$ of these limit measures equidistributes towards the Liouville measure. 
\end{abstract}

\maketitle

\section{Introduction}
Since the early work of Schnirelman \cite{Schni}, Colin de Verdi\`ere \cite{CdV} and Zelditch \cite{Zelditch1}, it is a well known fact that on any compact Riemannian
manifold $X$ whose geodesic flow is ergodic, one can find a full density sequence $\lambda_j\rightarrow +\infty$ of eigenvalues
of the Laplacian $\Delta_X$ such that the corresponding normalized eigenfunctions $\psi_j$ are equidistributed i.e. for all
$f\in L^2(X)$, we have 
$$\lim_{j\rightarrow +\infty}\int_X f(z)\vert \psi_j(z) \vert^2dv(z)=\int_X f(z)dv(z),$$
where $dv$ is the normalized volume measure.
Much less is known for non-compact manifolds and the closest analogs of the above celebrated theorem are due to 
Zelditch \cite{Zelditch2}, Luo-Sarnak \cite{LuoSarnak} and Jakobson \cite{Jakobson} in the case of non-compact finite area hyperbolic surfaces. Let us recall their results. By $\hh^2$ we denote the usual hyperbolic plane. Let $X=\Gamma\backslash \hh^2$ be a finite area surface where $\Gamma$ is a non co-compact
co-finite Fuchsian group. The non compact ends of $X$ are cusps which can be viewed on the universal cover
as fixed points $c_j$ in $\partial \hh^2$ of parabolic elements in $\Gamma$. The number of non-equivalent cusps is finite.
The spectrum of the Laplacian $\Delta_X$ consists of two types. The discrete part which corresponds to $L^2(X)$-eigenfunctions
and may be not finite as in the arithmetic case, but is conjectured by Philips-Sarnak \cite{PS} to be finite in the generic case.
The absolutely continuous part $[1/4,+\infty)$ which is parametrized ($t\in \rr$) by the
finite set of Eisenstein series $E_X(1/2+it;z,j)$ related to each cusp $c_j$. The Eisenstein series $E_X(1/2+it;z,j)$ are smooth
({\it non }$L^2(X)$) eigenfunctions
$$\Delta_X E_X(1/2+it;z,j) =(1/4+t^2)E_X(1/2+it;z,j).$$
In the Poincar\'e upper half-plane model, Eisenstein series are usually defined through the series 
(absolutely convergent for $\rea(s)>1$)
$$ E_X(s;z,j)=\sum_{\gamma \in \Gamma_j \backslash \Gamma} \left( \ima (\sigma_j^{-1}\gamma  z ) \right)^s,$$
which enjoys a meromorphic continuation to the whole complex plane and is analytic on the critical line $\{\rea(s)=1/2 \}$.
In the above formula, $\Gamma_j$ is the stabilizer of the cusp $c_j$ and $\sigma_j \in \mathrm{PSL}_2(\rr)$ is such that
$\sigma_j(\infty)=c_j$. In the case of the modular group $\mathrm{PSL}_2(\zz)$, there is only one Eisenstein series given by
the formula (valid for $\rea(s)>1$)
$$E_X(s;x+iy)=\frac{1}{2}\sum_{\gcd (c,d)=1}\frac{y^s}{\vert cz+d\vert^{2s}}.$$
For all $t\in \rr$, define the density $\mu_t$ by
$$\int_X a(z)d\mu_t(z):=\sum_j \int_X a(z)\vert E_X(1/2+it;z,j)\vert^2dv(z),$$
where $a \in C_0^\infty(X)$. If we assume that we have only finitely many eigenvalues (believed to be the generic case)
then Zelditch's equidistribution result \footnote{More details about the appropriate normalization can be found in his paper 
\cite{Zelditch2} where a microlocal statement is actually proved.} is as follows: for $a\in C_0^\infty(X)$, 
\[\frac{1}{s(T)}\int_{-T}^T \Big| \int_X ad\mu_t -\pl_ts(t) \int_X a\, dv\Big|dt\to  0\,\,\textrm{ as }T\to \infty\]
where $s(t)$ is the scattering phase, this coefficient appearing 
as a sort of regularization of Eisenstein series due to 
the fact that the Weyl law involves the continuous spectrum.

On the other hand, for the modular surface $X=\mathrm{PSL}_2(\zz)\backslash \hh^2$, Luo and Sarnak \cite{LuoSarnak} showed that 
as $t\rightarrow +\infty$,
$$\int_X a d\mu_t=\frac{48}{\pi}\log (t)\int_X a dv+o(\log(t)),$$
which is a much stronger statement obtained via sharp estimates on certain $L$-functions. A microlocal version was later
proved by Jakobson \cite{Jakobson}. 

In the present paper, we focus on the case of {\it infinite volume hyperbolic manifolds} without cusps, more precisely convex co-compacts quotients $X=\Gamma\backslash \hh^{n+1}$ of the hyperbolic space. Let us recall some basic notations.
A discrete group of orientation preserving isometries of $\hh^{n+1}$ is said to be convex co-compact
if it admits a polygonal, finitely sided fundamental domain whose closure does not intersect the limit set of $\Gamma$.
The limit set $\Lambda_{\Gamma}$ and the 
set of discontinuity $\Omega_\Gamma$ are usually defined by 
\[\Lambda_{\Gamma}:=\bbar{\Gamma.o}\cap S^n, \quad \Omega_\Gamma:=S^n\setminus \Lambda_\Gamma,\]
where $o\in\hh^{n+1}$ is a point in $\hh^{n+1}$.
By a result of Patterson and Sullivan \cite{Pat,Su}, the Hausdorff dimension of $\Lambda_\Gamma$  
\[\delta_\Gamma:=\dim_{\rm Haus} (\Lambda_\Gamma )\]
is also the exponent of convergence of the Poincar\'e series, i.e.  for all $m,m'\in \hh^{n+1}$ and $\la>0$,
\begin{equation}\label{delta}
\sum_{\gamma\in \Gamma} e^{-\la d(\gamma m,m')}<\infty \iff \la>\delta_\Gamma,
\end{equation}
 where $d(m,m')$ denotes the hyperbolic distance.

In that case the spectrum of $\Delta_X$ has been completely described by Lax-Philips and consists of the absolutely continuous spectrum $[n^2/4,+\infty)$ and a (possibly empty) finite set of eigenvalues in $(0,n^2/4)$. Just like in the finite volume case,
there is a way (see next $\S$ for full details) to define Eisenstein functions which parametrize the continuous spectrum. 
These functions replace in this setting the plane wave solutions $e^{\pm i\la z.\omega}$ of $(\Delta-\la^2)u=0$ in $\rr_z^n$ (here $\omega\in S^{n-1}$). 
Let $g$ denote the usual hyperbolic metric. Convex co-compact quotients $X=\Gamma\backslash \hh^{n+1}$ fall in the class of conformally compact manifolds, which means that $X$ compactifies smoothly to a manifold $\overline{X}$ with boundary $\partial \overline{X}$ such that there exists a smooth boundary defining 
function $x$ with the property that $x^2g$ extends to a smooth metric on $\bbar{X}$ and $|dx|_{x^2g}=1$ on $\pl X$. The boundary 
$\partial X$ is called the conformal boundary and can be identified with $\Gamma \backslash \Omega_\Gamma$. The 
Eisenstein series \footnote{They depend on the boundary defining function $x$; in the boundary variable $\xi$ they should be thought of as section of a complex power of the conormal bundle to the boundary} 
are smooth functions $E_X(s;m,\xi)$ where $s\in \cc$ is the spectral parameter as above
and $(m,\xi)\in X\times \partial \bbar{X}$. They are generalized non-$L^2(X)$ eigenfunctions of the Laplacian.
Using the ball model and an appropriate boundary defining function, Eisenstein series lift to $\hh^{n+1}$ to
$$E_X(s;m,\xi)=\sum_{\gamma \in \Gamma} \Big(\frac{1-|\gamma m|^2}{4|\gamma m-\xi|^2}\Big)^\la,$$
which is absolutely convergent for $\rea(s)>\delta_\Gamma$. Our first result is the following.
\begin{theorem}\label{eqinspace1}
Let $X=\Gamma\backslash\hh^{n+1}$ be a convex co-compact quotient with $\delta_{\Gamma}<n/2$.
Let $a\in C_0^\infty(X)$ and let $E_X(\la;\cdot,\xi)$ be an Eisenstein series as above with a given point $\xi\in \pl \bbar{X}$ at infinity. Then
 we have as $t\to + \infty$, 
\[\int_{X}a(m)\Big|E_X(\ndemi+it;m,\xi)\Big|^2dv(m) =  \int_X a(m)
E_X(n;m,\xi)dv(m)+\mc{O}(t^{2\delta_\Gamma-n}).\] 
\end{theorem}
This result shows individual equidistribution of high energy Eisenstein series. The limit measure
on $X$ is given by the harmonic density $E_X(n;m,\xi)$ whose boundary limit is the Dirac mass at $\xi \in \partial X$.

A microlocal extension of this theorem is actually proved in $\S 3$. We first need to introduce some adequate notations.
Fix any $\xi \in \partial X$. Let  ${\mathcal L}^\Gamma_\xi$ defined by 
$${\mathcal L}^\Gamma_\xi:=\overline{\cup_{\gamma\in \Gamma}\mc{L}_{\gamma \xi}}\subset S^*X,$$
where $\mc{L}_{\gamma\xi}$ are stable Lagrangian submanifolds of the unit cotangent bundle $S^*X$: 
the Lagrangian manifold $\mc{L}_{\gamma \xi}$ is defined to be the projection on $\Gamma\backslash S^*\hh^{n+1}$ of  
$$\{(m,\nu_{\gamma\xi}(m))\in S^*\hh^{n+1}; m\in \hh^{n+1}\},$$ 
where $\nu_{\gamma\xi}(m)$ is the unit (co)vector tangent to the geodesic starting at $m$ and pointing toward $\gamma\xi\in S^n$.  The set ${\mathcal L}^\Gamma_\xi$ ``fibers'' over $X$, and the fiber over a point $m\in X$ corresponds to the closure of the set of directions $v\in S^*X$ such that the geodesic starting at $m$ with directions $v$
converges to $\xi \in \pl\bbar{X}$ as $t\to +\infty$.
Since the closure of the orbit $\Gamma . \xi$ satisfies $\overline{\Gamma . \xi}\supset \Lambda_\Gamma$, 
we actually have
$$ {\mathcal L}^\Gamma_\xi \supset {\mathcal T}_+:=
\{ (m,\nu)\in S^*X\ :\ g_t(m,\nu)\ \textrm{remains bounded as}\ t\rightarrow +\infty\},$$
where $g_t:S^*X\rightarrow S^*X$ is the geodesic flow. The set ${\mathcal T}_+$ is often refered as the 
{\it forward trapped set}. The Hausdorff dimension of 
${\mathcal L}^\Gamma_\xi $ is $n+\delta_\Gamma+1$ and satisfies $n+1< \delta_\Gamma
+n+1<2n+1$ if $\Gamma$ is non elementary.


Our phase-space statement is the following, we refer the reader to $\S 3.2$ for the required background on pseudo-differential
operators.
\begin{theorem}
Let $A$ be a compactly supported 
$0$-th order pseudodifferential operator with principal symbol
\footnote{$a\in C_0^\infty(X,T^*X)$ means that $a\in C^\infty(T^*X)$ is compactly supported on the base 
but not necessarily in the fibers.} $a\in C_0^\infty(X,T^*X)$, then as $t\to +\infty$
\[\left\cjg AE_X(\ndemi+it;\cdot,\xi),E_X(\ndemi+it;\cdot,\xi)\right\cjd_{L^2(X)}=\int_{S^*X} a\,d\mu_{\xi}+\mc{O}(t^{-\min(1,n-2\delta_\Gamma)})\]
where $\mu_{\xi}$ is a $g_t$-invariant measure supported on the fractal subset ${\mathcal L}^\Gamma_\xi \subset S^*X$.
\end{theorem}

Notice that the fractal behaviour of the semi-classical limit $\mu_\xi$ can only be observed at the microlocal level. This is 
to our knowledge the first example in the mathematical literature of a high energy limit of eigenfunctions which is concentrated
microlocally on a genuine fractal trapped set.
\begin{figure}[ht!]
\begin{center}
\input{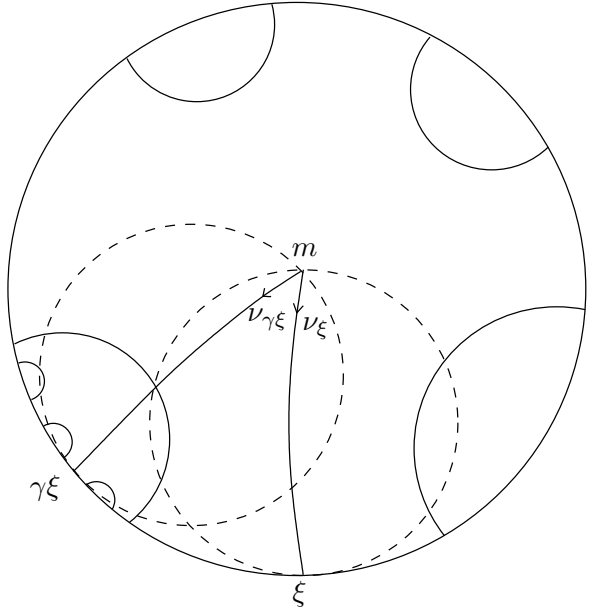}
\caption{Geodesics pointing toward $\xi$ and $\gamma\xi$ in the domain of discontinuity}
\end{center}
\end{figure}

In the theoretical physics literature, we mention some work of Ishio-Keating \cite{IsKe} on somehow related questions for billards.

By averaging over  the boundary with respect to the volume measure on $\partial\bbar{X}$ induced by the  defining function $x$, 
we obtain as $t\to +\infty$
\begin{equation}\label{averaging}
\int_{\pl\bbar{X}}\int_{X}a(m)\Big|E_X(\ndemi+it;m,\xi)\Big|^2dv(m)dv_{\pl\bbar{X}}(\xi) = 
{\rm vol}(S^n) \int_X a(m)dv(m)+\mc{O}(t^{2\delta_\Gamma-n})\end{equation}
and 
\[ \int_{\pl \bbar{X}}\left\cjg AE_X(\ndemi+it;\cdot,\xi),E_X(\ndemi+it;\cdot,\xi)\right\cjd_{L^2(X)} dv_{\pl\bbar{X}}(\xi)=\int_{S^*X} a\, d\mu+ \mc{O}(t^{-\min(n-2\delta_\Gamma,1)})
\]
where $\mu$ denotes the Liouville measure. This is the perfect analog of the previously known results for 
the modular surface (actually with a remainder in our case).

It is actually possible to prove a much stronger asymptotic expansion for the averaged case. We show the following.
\begin{theorem}
Let $X=\Gamma\backslash \hh^{n+1}$ be a convex co-compact hyperbolic manifold and assume that the limit set of 
$\Gamma$ has Hausdorff dimension $\delta_\Gamma<n/2$.
Let  $a\in C_0^\infty(X)$ and $\mc{C}$ be the set of all closed geodesics. 
Then for all $N\in \nn$, one has the expansion as $t\to +\infty$
\begin{equation}\label{qer}
\begin{gathered}
\int_{\pl X}\int_X a(m)|E_X(\ndemi+it;m,\xi)|^2dv(m)\,dv_{\pl \bbar{X}}(\xi)={\rm vol}(S^n)
\int_X a(m)dv(m)\\
+ L(t)\sum_{\gamma\in \mc{C}}
e^{-(\ndemi+it)\ell(\gamma)}\sum_{k=0}^{N}\frac{H_{k}(\ell(\gamma))}
{|\det({\rm Id}-e^{-\ell(\gamma)}R_\gamma^{-1})|}t^{-k}
\int_{\gamma} P_ka\,d\mu +\\
+ L(-t)\sum_{\gamma\in \mc{C}}
e^{-(\ndemi-it)\ell(\gamma)}\sum_{k=0}^{N}\frac{H_{k}(\ell(\gamma))}{|\det({\rm Id}-e^{-\ell(\gamma)}R_\gamma^{-1})|}(-t)^{-k}
\int_{\gamma} P_ka\, d\mu + \mc{O}(t^{-N-n-1}).
\end{gathered}
\end{equation}
where $H_k(\ell(\gamma))$ is an explicit bounded function of $\ell(\gamma)$, $P_k$ are differential operators of order 
$2k$ with coefficient uniformly bounded in terms of $\gamma\in\mc{C}$, $d\mu$ is the Riemannian measure induced on $\gamma$, $R_\gamma\in {\rm SO}(n)$ is the holonomy along $\gamma$ 
and 
\[\begin{gathered}
L(t)=t^{-(n-1)/2}\frac{2^{\ndemi+3-2it}|\Gamma(it)|^2}{\Gamma(it+\demi)\Gamma(\ndemi-it)}=\mc{O}(t^{-n})\\
 P_0= 1, \quad  H_{0}(\ell(\gamma))=2^\ndemi\pi^{\frac{n-1}{2}}e^{i\frac{n-1}{4}\pi}.
\end{gathered}\] 
\end{theorem}
Notice that the asymptotic expansion is valid at all orders and, except for the first term, involves oscillating sums over closed geodesics. This formula is similar to what physicists call ``the Gutzwiller formula" in the literature related to Quantum Chaos. In the mathematics literature it is of course quite similar to the Selberg trace formula (proved in \cite{PatPer,BO,GZw} in this setting, see also \cite{GuiNau}) 
and to trace formulae proved in \cite{CdV1,Cha, DuiGui,BrUr, Mei}.\\

Let us describe the organization of the paper. The section $\S 2$ is devoted to the necessary analytic and geometric background.
The proofs start in $\S 3$ and take advantage of the convergence of Poincar\'e series garanteed by the hypothesis
$\delta_\Gamma <n/2$. We first prove the ``configuration space" equidistribution result using a direct approach which
is based on an ``off-diagonal" cancelation phenomena due to some non-stationnary phase effect. To get a sharp remainder estimates in terms of $\delta$, we need to analyze differently 
the elements in the group with large and small translation lengths. 
This technique is robust enough to extend to the pseudo-differential (phase space) equidistribution case. We treat the problem using semi-classical analysis, namely
some standard results on the action of $h$-Pseudos on $h$-Lagrangian distributions. Theorem \ref{averaging} is based on
a completely different technique and make use of Stone's formula and the high frequency asymptotic for the resolvent
on the universal covering $\hh^{n+1}$. 
The result follows from a repeated application of stationnary phase expansions (with parameter) and requires some 
precise estimate in terms of the group elements.

We believe that extensions of these results to non-constant negative curvatures setting can 
be obtained for manifolds with asymptotically 
hyperbolic (or asymptotically Euclidean) ends under a condition on the topological pressure of the flow on the trapped set like in \cite{NonZw}. This will be pursued elsewhere. \\

\textbf{Acknowledgement.} We both thank D. Jakobson for pointing out previous works on equidistribution of Eisenstein series and S. Nonnenmacher for helpful discussions on that topic, 
as well as M. Zworski for his interest and urging us to write this down.  
Both authors are supported by ANR grant ANR-09-JCJC-0099-01.

\section{Generalities}

\subsection{Definition of Eisenstein series}
Let $X=\Gamma\backslash \hh^{n+1}$ be a convex co-compact hyperbolic manifold. 
The group $\Gamma$ is a discrete group of orientation preserving 
isometries of $\hh^{n+1}$ with only hyperbolic transformations (i.e. fixing $2$ points on the sphere 
$S^n=\pl\hh^{n+1}$). The limit set is $\Lambda_\Gamma :=\bbar{\Gamma.o}\cap S^n$ where $o\in\hh^{n+1}$ is the center in the unit ball model $B_o(1)\subset \rr^{n+1}$ of $\hh^{n+1}$. The discontinuity set is the complement $\Omega_\Gamma:=S^n\setminus\Lambda_\Gamma$ and $\Gamma$ acts properly discontinously on $\hh^{n+1}\cup \Omega_\Gamma\subset \bbar{B_o(1)}$.
The manifold $X$ compactifies smoothly to a manifold $\bbar{X}$ with boundary $\pl \bbar{X}$ by setting 
$\bbar{X}=\Gamma\backslash(\hh^{n+1}\cup \Omega_\Gamma)$. 
If $g$ denotes the hyperbolic metric on $X$, there exists a smooth boundary defining 
function $x$ such that $x^2g$ extends to a smooth metric on $\bbar{X}$ and that $|dx|_{x^2g}=1$ 
on $\pl\bbar{X}$.  
The Laplacian $\Delta$ on $(X,g)$ has absolutely continuous spectrum on $[n^2/4,\infty)$ and a possibly non-empty finite set of eigenvalues in $(0,n^2/4)$.\\
 By  \cite{MM,GZ}, the resolvent of the Laplacian
\[R_X(\la):=(\Delta-\la(n-\la))^{-1}  \quad \textrm{ defined in the half plane }\rea(\la)>n/2 \]
admits a meromorphic continuation to the whole complex plane $\cc$, with poles of finite rank (i.e. the polar part at a 
pole is a finite rank operator), as a family of bounded operators 
$$R_X(\la): C_0^\infty(X)\to C^\infty(X).$$ 
Using \cite{MM,GZ},
the resolvent integral kernel $R_X(\la;m,m')$ near the boundary $\pl\bbar{X}$ 
has an asympotic expansion given as follows: for any $m\in X$ fixed
\[  m'\to R_X(\la;m ,m' )x(m')^{-\la} \in  C^\infty(\bbar{X})\]
and similarly for $m'\in X$ fixed and $m'\to \pl\bbar{X}$.
The Eisenstein series are functions in $C^\infty(X\x \pl\bbar{X})$ defined by\footnote{The normalization factor 
is not the usual one, but it is the natural one for our equidistribution purpose; this corresponds somehow
to the normalization given by the Weyl law in the co-compact setting.} 
\begin{equation}\label{definitionE}
E_X(\la;m,\xi)=\frac{2\pi^{\ndemi}\Gamma(\la-\ndemi+1)}{2^{-\la}\Gamma(\la)}[x(m')^{-\la}R_X(\la;m,m')]|_{m'=\xi}, \quad  \xi\in \pl\bbar{X}.
\end{equation}
The terminology `series' will become clear later since they can be obtained by 
series over the group. 
Notice that $E_X(\la;\cdot,\xi)$ depends a priori on the choice of the boundary defining function $x$ however
any other choice $x'=\psi(\xi)x+O(x)$ with $\psi\in C^\infty(\pl\bbar{X})$  
would simply amount to multiply $E_X(\la;m,\xi)$ by $\psi(\xi)^{-\la}$. 
Moreover, they are generalized  eigenfunctions of $\Delta_X$
\[(\Delta_X-\la(n-\la))E_X(\la;\cdot,\xi)=0, \quad \forall \xi\in \pl\bbar{X}.\]
The functions $E_X(\la,\cdot,\xi)$ replace the plane wave functions $z\to e^{i\la z.\xi}$ in $\rr^n$ (with $\xi\in S^{n-1}$)
which satisfy $(\Delta_{\rr^n}-\la^2)e^{i\la z.\xi}=0$, they provide a spectral representation of the continous spectrum for $\Delta_X$.

Let us define the constant   
\begin{equation}\label{cla}
C(\la)=\pi^{-\ndemi}2^{-\la}\frac{\Gamma(\la)}{\Gamma(\la-\ndemi)}.
\end{equation}

The spectral measure  $2t d\Pi_X(t) $ of the Laplacian $\Delta_X$ is related to  Eisenstein series by
\begin{equation}\label{stone}
\begin{split}
2td\Pi_X(t;m,m') =& \frac{is}{\pi}(R_X(\ndemi +it;m,m')-R_X(\ndemi -it;m,m'))\\
=&\frac{|C(\ndemi+it)|^2}{2\pi }\int_{\pl \bbar{X}}E_X(\ndemi+it;m,\xi)E_X(\ndemi -it;m',\xi)dv_{\pl\bbar{X}}(\xi)\\
2td\Pi_X(t;m,m)=& \frac{|C(\ndemi+it)|^2}{2\pi }\int_{\pl \bbar{X}}|E_X(\ndemi+it;m,\xi)|^2\, dv_{\pl\bbar{X}}(\xi)
\end{split}
\end{equation}
with $dv_{\pl\bbar{X}}$ the Riemannian measure on $\pl\bbar{X}$ induced by 
the metric $h_0=(x^2g)|_{T\pl X}$. The first identity is Stone's formula, 
the second identity is an application of Green's formula proved in \cite[Prop .1]{Gui} or \cite{Per}.
We notice that the density $|E_X(\ndemi+it;m,\xi)|^2\, dv_{\pl\bbar{X}}(\xi)$ is independent of the choice of 
boundary defining function $x$ used for the definition \eqref{definitionE}, as long as $h_0$ is defined 
by $x^2g|_{T\pl\bbar{X}}$.\\

\subsection{Eisenstein functions on $\hh^{n+1}$}
To fix the ideas, and since it will be used later, let us recall the case of $\hh^{n+1}$.
The resolvent kernel is given (away from the diagonal $m=m'$) by \cite[Sec. 2]{GZ}:
\begin{equation}\label{R0}
\begin{split}
R_{\hh^{n+1}}(\la; m,m')=&\frac{\pi^{-\ndemi} 2^{-2\la-1}\Gamma(\la)}{\Gamma(\la-\ndemi+1)}\cosh^{-2\la}\Big(\frac{d(m,m')}{2}\Big)\\
& \x F\Big(\la,\la-\frac{n-1}{2},2\la-n+1; \cosh^{-2}\Big(\frac{d(m,m')}{2}\Big)\Big)
\end{split}
\end{equation}
where $d(m,m')$ is the hyperbolic distance between $m$ and $m'$, and for $z\in \cc\setminus [1,+\infty)$,
$$F(a,b,c;z)=\frac{\Gamma(c)}{\Gamma(b)\Gamma(c-b)}\int_0^1t^{b-1}(1-t)^{c-b-1}(1-tz)^{-a}dt,$$ 
is the usual hypergeometric function. In the ball model 
$B_o(1):=\{m\in \rr^{n+1}; |m|<1\}$ (with center denoted by $o$), the function  
\begin{equation} \label{x(m)}
x(m):=2e^{-d(m,o)}=2\frac{1-|m|}{1+|m|}
\end{equation}
is a boundary defining function for the closed ball $\bbar{B_o(1)}$.
Moreover, using 
\begin{equation}\label{sinh}
\sinh^2\Big(\frac{d(m,m')}{2}\Big)=\frac{|m-m'|^2}{(1-|m|^2)(1-|m'|^2)},
\end{equation}
we obtain for $E_0(\la;m,\xi):=E_{\hh^{n+1}}(\la;m,\xi)$ defined with the boundary defining function $x$
\begin{equation}\label{E_0}
E_0(\la;m,\xi)=\Big(\frac{1-|m|^2}{4|m-\xi|^2}\Big)^\la=e^{\la \phi_\xi(m)} \quad \textrm{ with }\,\, \phi_\xi(m):=\log\Big(\frac{1-|m|^2}{4|m-\xi|^2}\Big).
\end{equation}  
Let $\la_h:=n/2+i/h$ with $h>0$ small, then as a function of $m$, $E_0(\la_h;m,\xi)$ is a semi-classical Lagrangian distribution associated to the Lagrangian  submanifold  
\[\mc{L}_\xi:= \{ (m,d\phi_\xi(m))\in T_m^*\hh^{n+1}; m\in \hh^{n+1}\}\subset T^*\hh^{n+1}.\]
If $H_\xi(m)$ is the horosphere tangent to $\xi$ and containing $m$, then
$\phi_\xi(m)=\pm d(H_\xi(m),H_\xi(o))$ is the hyperbolic distance between $H_{\xi}(m)$ and $H_{\xi}(o)$. 
Moreover, since $\phi_{R\xi}(m)=\phi_{\xi}(R^{-1}m)$ for any 
$R\in {\rm SO}(n+1)$, the norm $|d\phi_\xi(m)|_{g}$ with respect to the hyperbolic metric 
$g$ is independent of $\xi$. By going through the half-space model $\rr^+_{x}\x\rr^n_{y}$ of $\hh^{n+1}$, 
and choosing $\xi=\infty$, we have $\phi_{\infty}(x,y)=\log(y)$ and thus $|d\phi_\infty(m)|_g=1$ for all $m\in\hh^{n+1}$, which implies that 
\begin{equation}\label{dphi}
|d\phi_\xi(m)|_g=1  \textrm{ for all } \xi\in S^n,m\in\hh^{n+1}.
\end{equation}
The level sets of $\phi_\xi$ are the horospheres tangent to $\xi$, and the covector 
$d\phi_\xi(m)$
is dual via the metric to the gradient $\nabla\phi_\xi(m)$ which is the
unit vector tangent to the geodesic going from $m$ to $\xi$.  The manifold $\mc{L}_\xi$ 
is the \emph{stable manifold} associated to $\xi$. 

Let us describe the action of isometries on Eisenstein funtions.
\begin{lemma}\label{Dgamma}
Let $\gamma\in {\rm Isom}_+(\hh^{n+1})$ be an isometry of $\hh^{n+1}$, then for all $m\in\hh^{n+1}$ and $\xi\in S^n$,
\[ E_0(\la; \gamma m,\gamma\xi)=E_0(\la;m,\xi) |D\gamma(\xi)|^{-\la}\]
where $|D\gamma(\xi)|$ is the Euclidean norm of $D\gamma$ at the point $\xi\in S^n$. 
\end{lemma}
\begin{proof}
We use the fact that the function \eqref{sinh} is invariant with respect to $\gamma$, i.e. 
for  all $m,m'$
\[ \frac{|\gamma m-\gamma m'|}{(1-|\gamma m'|^2)(1-|\gamma m'|^2)}=\frac{|m-m'|^2}{(1-|m|^2)(1-|m'|^2)}\]
and thus multiplying this inequality by $(1-|m'|^2)$ and letting $m'\to \xi\in S^n$, we obtain
$E_0(\la;\gamma m,\gamma \xi)/E_0(\la;m,\xi)=\lim_{m'\to \xi}(\frac{1-|\gamma m'|^2}{1-|m'|^2})^{-\la}$.
This limit is given by $|D\gamma(\xi)|^{-\la}$. 
\end{proof}
Notice that for any M\"obius transformation $\gamma$, 
$|D\gamma(\xi)v|=|D\gamma(\xi)|$ for any $v$ of Euclidean norm $1$. Let us finally give the formula for the 
integral of the hyperbolic Poisson kernel: for all $m\in \hh^{n+1}$, 
\begin{equation}\label{poissonkernel}
\int_{S^n}E_0(n;m,\xi)d\xi =1/C(n)={\rm vol}(S^n).
\end{equation}
where $C(n)$ is defined in \eqref{cla} and $d\xi$ is the canonical measure on $S^n$.

\subsection{Eisenstein series on $\Gamma\backslash \hh^{n+1}$}
\begin{lemma}\label{quotient}
Let $\Gamma$ be a convex co-compact group of isometries of $\hh^{n+1}$ with $\delta_\Gamma<n/2$ and 
$\pi_\Gamma:=\hh^{n+1}\to \Gamma\backslash\hh^{n+1}$ be the quotient map.
Let $\xi\in \Omega_{\Gamma}$ and $m\in\hh^{n+1}$, then for $\rea(\la)>\delta_\Gamma$ the series
\[\til{E}(\la;m,\xi):=\sum_{\gamma\in \Gamma}E_{\hh^{n+1}}(\la;\gamma m,\xi)\]  
converges absolutely and is a $\Gamma$-automorphic function on $\hh^{n+1}$ satisfying
\[E_X(\la;\pi_\Gamma(m),\pi_\Gamma(\xi))=\til{E}(\la;m,\xi)\] 
if $E_X(\la;\cdot,\cdot)$ is defined 
by \eqref{definitionE} with a boundary defining function $x$ so that, in a neighbourhood of  $\xi$, 
$\pi_\Gamma^{*}(x)=(1-|m|)/(1+|m|)$. 
\end{lemma}
\begin{proof}
By combining \eqref{x(m)} and \eqref{E_0}, together with the estimate 
$|\gamma m-\xi|>\eps>0$ for some $\eps$ uniform in $\gamma\in \Gamma$, we deduce the existence of a constant $C_\la$ uniform in $\gamma$ so that
\[|E_0(\la;\gamma,\xi)|\leq C_\la e^{-\la d(\gamma m,o)}.\]
This shows that the series converges in $\rea(\la)>\delta_\Gamma$ by \eqref{delta}.

The resolvent kernel for $\rea(\la)>n/2$ on the quotient is given, for $m,m'\in \hh^{n+1}$ so that $\pi_\Gamma(m)\not=\pi_\Gamma(m')$, by the sum 
\[R_X(\la;\pi_\Gamma(m),\pi_\Gamma(m'))=\sum_{\gamma\in \Gamma}R_{\hh^{n+1}}(\gamma m,m'),\]
and this series converges absolutely in $\rea(\la)>\delta_\Gamma$ by \eqref{R0} and \eqref{delta}, moreover  the convergence is uniform on compact sets in the variable $m,m'$ (in any $C^k$ norms). 
The proof is standard and is left to the reader.
By using \eqref{x(m)} and \eqref{sinh}, it is easily seen that after we multiply this last equation by $x(\pi_\gamma(m'))^{-\la}$, we take the limit $m'\to \xi$ and it is given by 
\[\frac{2\la-n}{C(\la)}\sum_{\gamma\in \Gamma}\lim_{m'\to \xi}(R_{\hh^{n+1}}(\la;\gamma m,m')(2\tfrac{1-|m'|}{1+|m'|})^{-\la})=\til{E}(\la;m,\xi).\]
This concludes the proof.
\end{proof}

\subsection{The Liouville measure}

If $(X,g)$ is an oriented $n+1$ dimensional Riemannian manifold, recall that the Liouville form (or measure) on the unit sphere bundle $SX$
is given by the $2n+1$ form $\mu:=\theta\wedge d\theta^n$ where $\theta$ is the contact form on $SX$
\[\theta_{u}(v):=g(u,\pi_*v) \]
if $\pi:SX\to X$ is the projection.
Isometries $\gamma$ of $X$ act on $SX$ by $\gamma.(m,v):=(\gamma m,d\gamma(m)v)$ and clearly 
$\gamma^*\theta=\theta$. In particular, this implies that $\gamma^*\mu=\mu$. Since 
the group ${\rm Isom}_+(\hh^{n+1})$ acts transitively on $S\hh^{n+1}$, the Liouville form is determined by its 
value at one point, for instance $(m,v)=(o,e_1)$ where $e_1$ is a fixed unit vector of $S_o\hh^{n+1}$. 

Let us define the map $\psi:\hh^{n+1}\x S^n\to S\hh^{n+1}$ by 
\begin{equation}\label{identif}
\psi: (m,\xi) \mapsto (m, \nabla \phi_\xi(m))
\end{equation} 
where $\nabla$ denotes the gradient with respect to the hyperbolic metric $g$.
This map is a diffeomorphism and it identifies the point $(m,v)\in S\hh^{n+1}$ to the point $(m,\xi)\in \hh^{n+1}\x S^n$
so that $\xi$ is the end point of the geodesic starting at $m$ with tangent vector $v$.   
The group of isometries ${\rm Isom}_+(\hh^{n+1})$ acts transitively on $\hh^{n+1}\x S^n$ by $\gamma.(m,\xi):=(\gamma m,\gamma\xi)$ and using Lemma \ref{Dgamma}, we obtain for all $\gamma\in{\rm Isom}_+(\hh^{n+1})$
\begin{equation}\label{gammanu}
\gamma^*\nu=\nu ,\quad \textrm{ where } \nu:=E(n;m,\xi)dv(m)\wedge d\xi 
\end{equation}
where $dv(m)$ and $d\xi$ are the canonical volume forms of $\hh^{n+1}$ and $S^n$. But we also have 
$\psi(\gamma m,\gamma\xi)=\gamma.\psi(m,\xi)$
and therefore we deduce that for all $\gamma\in {\rm Isom}_+(\hh^{n+1})$
\[\gamma^*\psi^*\mu=\psi^*\gamma^*\mu=\psi^*\mu.\]
By \eqref{gammanu}, this implies that both $\nu$ and $\psi^*\mu$ are preserved by the transitive action of ${\rm Isom}_+(\hh^{n+1})$,
and thus there exists a constant $C_n$ such that $\psi^*\mu$, and an easy computation at $m=0,\xi=e_1$ gives 
\begin{equation}\label{psimu}
\psi^*\mu=\nu.
\end{equation}
The same is obviously true with the Liouville measure on $S^*\hh^{n+1}$ by duality, provided $\psi$ is defined 
with $d\phi_\xi(m)$ instead of $\nabla\phi_\xi(m)$.
On a quotient $X=\Gamma\backslash\hh^{n+1}$, we also have $S^*X=\Gamma\backslash S^*\hh^{n+1}\simeq 
\Gamma\backslash (\hh^{n+1}\x S^n)$ and  the Liouville measure on the quotient lifts to the Liouville measure 
$\mu$ on $\hh^{n+1}$.

\section{Proof of equidistribution of Eisenstein series}

\subsection{Equidistribution on $X$} 
In this section, we show equidistribution in space of the Eisenstein series.
\begin{theorem}\label{eqinspace}
Let $X=\Gamma\backslash\hh^{n+1}$ be a convex co-compact quotient with $\delta_{\Gamma}<n/2$, and let $x$ be a boundary defining function of $\bbar{X}$.
Let $a\in C_0^\infty(X)$ and let $E_X(\la;\cdot,\xi)$ be the Eisenstein series defined by \eqref{definitionE} using the defining function $x$ and a point $\xi\in \pl \bbar{X}$ at infinity. Then
 we have as $t\to \infty$, 
\[\int_{X}a(m)\Big|E_X(\ndemi+it;m,\xi)\Big|^2dv(m) = \int_X a(m)
E_X(n;m,\xi)dv(m)+\mc{O}(t^{2\delta-n}).\] 
\end{theorem}
\begin{proof}
Let us consider the ball model  $B_o(1)=\{m\in \rr^{n+1}; |m|<1\}$
of $\hh^{n+1}$, with boundary the sphere $S^n$.  Consider a fundamental domain $\mc{F}$ of $\Gamma$ and 
let $\xi\in \Omega_\Gamma\cap \bbar{F}$ be a point in the domain of discontinuity. The boundary defining function $x$ 
of $X$ lifts to an automorphic function $\til{x}$ on $\hh^{n+1}$ and we define 
$\eta(\xi):=2\lim_{m\to \xi}(1-|m|)/\til{x}(m)$.  
By Lemma \ref{quotient}, the Eisenstein series on $X=\Gamma\backslash \hh^{n+1}$ at parameter $\la$ such that ${\rm Re}(\la)>\delta_\Gamma $  is given by 
\begin{equation}\label{defEis}
\begin{gathered}
E_X(\la; \pi_\Gamma(m),\pi_\Gamma(\xi))=\sum_{\gamma\in \Gamma} E_{0}(\la;\gamma m,\xi)\eta(\xi)^{\la}=\sum_{\gamma\in \Gamma} \Big(E_{0}(1;\gamma m,\xi)\eta(\xi)\Big)^\la \\
E_0(1;m,\xi):= \frac{1-|m|^2}{4|m-\xi|^2}, \end{gathered}
\end{equation}
for $m\in \hh^{n+1}$. For simplicity, we shall identify $\pi_\Gamma(m)$ with $m$ in the fundamental domain.
We are interested in the pointwise norm (for $t\in \rr$)
\[\begin{split}
\Big|E_X(\ndemi+it; m,\xi)\Big|^2= &\sum_{\gamma\in \Gamma} \Big(E_0(1;\gamma m,\xi)\eta(\xi)\Big)^{n}\\
&+\eta(\xi)^n\sum_{\gamma\not=\gamma'} \Big(E_0(1;\gamma m,\xi)E_0(1;\gamma'm,\xi)\Big)^{\ndemi} 
\Big(\frac{E_0(1;\gamma m,\xi)}{E_0(1;\gamma'm,\xi)}\Big)^{it}\\
\Big|E_X(\ndemi+it; m,\xi)\Big|^2=& E_X(n;m,\xi)+\eta(\xi)^n\sum_{\gamma\not=\gamma'} \Big(E_0(1;\gamma m,\xi)E_0(1;\gamma'm,\xi)\Big)^{\ndemi} 
\Big(\frac{E_0(1;\gamma m,\xi)}{E_0(1;\gamma'm,\xi)}\Big)^{it}
\end{split}\] 
Let $a\in C_0^\infty(X)$ and consider the off-diagonal sum $\sum_{\gamma\not=\gamma'}I_{\gamma,\gamma'}(t)$ where
\begin{equation}\label{offdiag}
I_{\gamma,\gamma'}(t):=\int_{X}a(m) \Big(E_0(1;\gamma m,\xi)E_0(1;\gamma'm,\xi)\Big)^{\ndemi} 
\Big(\frac{E_0(1;\gamma m,\xi)}{E_0(1;\gamma'm,\xi)}\Big)^{it}dv(m)
\end{equation}
To end the proof, we need to show the 
\begin{lemma}\label{Idecays}
As $t\to \infty$, the off-diagonal sum satisfies 
$\sum_{\gamma\not=\gamma'}I_{\gamma,\gamma'}(t)=\mc{O}(t^{2\delta_\Gamma-n})$
\end{lemma}
\begin{proof}
The proof reduces to some non-stationary phase in $t$ on a finite number of terms in each factor of 
$\Gamma\x \Gamma$, essentially those $\gamma$ which satisfy $d(o,\gamma o)\leq \log t$.
Let us first define the Buseman function
\begin{equation}\label{buseman}
B_\xi(m,m'):=\log(1-|m|^2)-\log(|m-\xi|^2)-\log(1-|m'|^2)+\log(|m'-\xi|^2)
\end{equation}
which satisfies the cocycle condition $B_\xi(m,m'')=B_\xi(m,m')+B_\xi (m',m'')$ and 
for any $\gamma\in {\rm Isom}(\hh^{n+1})$
\[B_{\gamma\xi}(\gamma m,\gamma m')=B_\xi(m,m').\]
One can rewrite \eqref{offdiag} as 
\[\int_{X}a(m) \Big(E_0(1;\gamma m,\xi)E_0(1;\gamma'm,\xi)\Big)^{\ndemi} 
e^{itB_\xi(\gamma m,\gamma' m)}dv(m)
\] 
The phase $B_\xi(\gamma m,\gamma' m)=B_{\gamma^{-1}\xi}(m,\gamma^{-1}o)+ B_{{\gamma'}^{-1}\xi}({\gamma'}^{-1}o,m)$
has for gradient (with respect to Euclidean metric) in $m$ 
\[
\nabla_{m}B_\xi(\gamma m,\gamma' m)=2\Big(\frac{m-\gamma^{-1}\xi}{|m-\gamma^{-1}\xi|^2}-
\frac{m-{\gamma'}^{-1}\xi}{|m-{\gamma'}^{-1}\xi|^2}\Big)
\]
with Euclidean norm given by 
\begin{equation}\label{euclidean}
\Big|\nabla_{m}B_\xi(\gamma m,\gamma' m)\Big|=\frac{2|\gamma^{-1}\xi-{\gamma'}^{-1}\xi|}{|m-{\gamma'}^{-1}\xi|.|m-\gamma^{-1}\xi|}.
\end{equation}
Remark that $\nabla_mB_\xi(\gamma m,\gamma' m)=Q_m(\gamma^{-1}\xi,\gamma^{-1}\xi')/(|m-{\gamma'}^{-1}\xi|^2.|m-\gamma^{-1}\xi|^2)$ for some $Q_m(u,v)$ polynomial in $(u,v)$ and with $Q_m(u,u)=0$, moreover it is smooth in $m$
in a compact set $K\subset \hh^{n+1}$, we then easily deduce that  for each $\alpha\in \nn^{n+1}$ there is $C_\alpha>0$ such that for all $m\in K$ and $\alpha\in \nn^{n+1}$
\begin{equation}\label{nablaB}
\Big|\pl^\alpha_m \nabla_{m}B_\xi(\gamma m,\gamma' m)\Big|\leq C_\alpha|\gamma^{-1}\xi-{\gamma'}^{-1}\xi|.
\end{equation}
Let $o$  be the center of the unit ball in $\rr^{n+1}$.
For $\gamma$  a M\"obius transformation preserving the unit ball, we 
define  $a_\gamma=\gamma^{-1}(\infty) \in\rr^{n+1}$ the center of the isometric sphere $S(a_\gamma,r_\gamma)$ 
of $\gamma$, where $r_\gamma=1/\sinh(\tfrac{1}{2}d(o,\gamma o))$ is the radius of $S(a_\gamma,r_\gamma)$ (here $d(\cdot,\cdot)$ is the hyperbolic distance). By \S 3.5 in Beardon \cite{Bear}, we have 
 for any $\xi,\xi'\in \bbar{B_o(1)}$
\begin{equation}\label{formulebeardon} 
|\gamma \xi-\gamma \xi'|=\frac{1}{\sinh^2(\tfrac{1}{2}d(o,\gamma o))} \frac{|\xi-\xi'|}{|\xi -a_{\gamma}| \, |\xi'-a_\gamma|}.
\end{equation}
Since $|a_\gamma|$ is bounded uniformly, we can set $\xi'=\gamma^{-1}\circ \gamma' \xi$ and we
 deduce directly that $ |\gamma \xi-\gamma' \xi|^2\geq Ce^{-d(o,\gamma o)}$ and by symmetry in $\gamma,\gamma'$ this gives 
 \begin{equation} \label{boundbelow1}
 |\gamma \xi-\gamma' \xi|\geq C\max (e^{-d(o,\gamma o)},e^{-d(o,\gamma' o)})
 \end{equation} 
uniformly in $\gamma\not= \gamma'$ and $\xi\in \Omega_\Gamma$.
By combining \eqref{euclidean}, \eqref{nablaB}and \eqref{boundbelow1}, we deduce that for $m$ in a  compact set $K$ of $\hh^{n+1}$ one has 
\begin{equation}\label{boundbelow2}
\Big|\pl^\alpha_m \frac{\nabla_{m}B_\xi(\gamma m,\gamma' m)}{|\nabla_{m}B_\xi(\gamma m,\gamma' m)|^2}\Big| 
\leq C_\alpha \min (e^{d(o,\gamma o)},e^{d(o,\gamma' o)}).
\end{equation}
For each $\gamma,\gamma' \in \Gamma$, we can therefore integrate by part to get 
\[\begin{split}
I_{\gamma,\gamma'}(t)=& \frac{1}{it}\int_{X}a(m) \Big(E_0(1;\gamma m,\xi)E_0(1;\gamma'm,\xi)\Big)^{\ndemi}
\nabla_m(e^{it B_{\xi}(\gamma m,\gamma'm)}). \frac{\nabla_{m}B_\xi(\gamma m,\gamma' m)}{|\nabla_{m}B_\xi(\gamma m,\gamma' m)|^2}dv(m)\\
=& \frac{1}{it}\int_{X}e^{it B_{\xi}(\gamma m,\gamma'm)}\nabla^*\Big(a(m)\Big(E_0(1;\gamma m,\xi)E_0(1;\gamma'm,\xi)\Big)^{\ndemi}
\frac{\nabla_{m}B_\xi(\gamma m,\gamma' m)}{|\nabla_{m}B_\xi(\gamma m,\gamma' m)|^2}\Big)dv(m)
\end{split}\]
where $\nabla^*$ is the adjoint of $\nabla$ with respect to the Riemannian measure $dv(m)$
For $\xi\in \Omega_{\Gamma}$, there exists $C>0$ depending only on $K$ such that for 
all $m\in K$ and $\gamma\in \Gamma$, one has 
\[|\gamma m-\xi|>C , \quad (1-|\gamma m|^2)\leq Ce^{-d(o,\gamma o)}, \quad
|D\gamma(m)|\leq Ce^{-d(o,\gamma o)}.\]
The  last identity comes for instance from \eqref{formulebeardon}. Combining these estimates with 
\eqref{boundbelow2}, we obtain that there is $C>0$ such that for all $\gamma\not=\gamma'$ and $m\in K$
\[\begin{gathered}
\Big|\nabla^*\Big(a(m)\Big(E_0(1;\gamma m,\xi)E_0(1;\gamma'm,\xi)\Big)^{\ndemi}
\frac{\nabla_{m}B_\xi(\gamma m,\gamma' m)}{|\nabla_{m}B_\xi(\gamma m,\gamma' m)|^2}\Big)\Big|\leq \\
Ce^{-\ndemi (d(o,\gamma o)+d(o,\gamma' o))}\min(e^{d(o,\gamma o)},e^{d(o,\gamma'o)}).
\end{gathered}\] 
We thus have the estimate, which is uniform in $\gamma\not=\gamma'$
\[ |I_{\gamma,\gamma'}(t)|\leq C t^{-1}e^{-\ndemi (d(o,\gamma o)+d(o,\gamma' o))}\min(e^{d(o,\gamma o)},e^{d(o,\gamma'o)}).\]
We can iterate this argument and get for $j\in \nn$
\begin{equation}\label{jtimes}
|I_{\gamma,\gamma'}(t)|\leq C t^{-j}e^{-\ndemi (d(o,\gamma o)+d(o,\gamma' o))}\min(e^{jd(o,\gamma o)},e^{jd(o,\gamma'o)}).\end{equation}
Let us decompose $\Gamma$ into  
\[\Gamma=\Gamma_{\leq}\sqcup \Gamma_{>}:\{\gamma\in \Gamma; d(o,\gamma)\leq \log(t)\}\sqcup   \{\gamma\in \Gamma; d(o,\gamma)>\log(t)\}\]
and then the double sum $\sum_{(\gamma,\gamma')\in \Gamma\x \Gamma\setminus {\rm diag}}
I_{\gamma,\gamma'}(t)$ into $4$ parts involving $\Gamma_\alpha\x \Gamma_\beta\setminus{\rm diag}$ for $\alpha,\beta\in \{\leq ,>\}$. For the elements in $\Gamma_\leq$, we can now take advantage of the bound 
\eqref{jtimes} with $j$ taken larger than $n/2-\delta_\Gamma$.
Using the counting estimate (proved by Patterson \cite{PatArx})
\begin{equation}\label{pater}
\sharp\{ \gamma\in \Gamma; d(o,\gamma o)\leq T\}=\mc{O}(e^{\delta_\Gamma T})
\end{equation}
it is easily seen that for each $\alpha,\beta\in \{\leq ,>\}$, we have 
\[ \sum_{(\gamma,\gamma')\in \Gamma_{\alpha}\x \Gamma_{\beta}\setminus {\rm diag}}I_{\gamma,\gamma'}(t)=\mc{O}(t^{2\delta_\Gamma-n}).\]
Let us for instance write the case $\alpha=\beta=\leq$, then by \eqref{jtimes} and \eqref{pater} we have for 
any $j>n/2-\delta_\Gamma$ integer
\[\begin{split}
\sum_{\substack{d(o,\gamma o)\leq \log t\\ d(o,\gamma' o)\leq \log t, \gamma'\not=\gamma}}I_{\gamma,\gamma'}(t)=&
\mc{O}(t^{2j})\sum_{d(o,\gamma o)\leq \log  t}e^{(j-\ndemi) d(o,\gamma o)}
\sum_{d(o,\gamma' o)\leq \log  t}e^{(j-\ndemi) d(o,\gamma' o)}\\
=&\mc{O}(t^{n-2\delta_\Gamma})
\end{split}\]
and the other cases are similar.
\end{proof}
This achieves the proof of the Theorem. 
\end{proof}
Notice that $E_X(n;m,\xi)$ is a harmonic function, more precisely this is the harmonic function on $X$ with distributional boundary value the Dirac mass $\delta_\xi$ at $\xi$.

As a corollary, we obtain 
\begin{cor}\label{firstcor}
Let $(X,g)=\Gamma\backslash\hh^{n+1}$ be a convex co-compact quotient with $\delta_{\Gamma}<n/2$, let $x$ be a boundary defining function of $\bbar{X}$ and let $dv_{\pl \bbar{X}}$ be the Riemannian measure 
on the conformal boundary induced by the smooth metric $x^2g|_{\pl\bbar{X}}$. 
Let $a\in C_0^\infty(X)$ and let $E_X(\la;\cdot,\xi)$ be the Eisenstein series defined by \eqref{definitionE} using the defining function $x$ and a point $\xi\in \pl \bbar{X}$ at infinity. Then, 
 we have as $t\to \infty$
\[\int_{\pl\bbar{X}}\int_{X}a(m)\Big|E_X(\ndemi+it;m,\xi)\Big|^2dv(m)dv_{\pl\bbar{X}}(\xi) = 
{\rm vol}(S^n) \int_X a(m)dv(m)+\mc{O}(t^{2\delta_\Gamma-n}).\]
\end{cor}
\begin{proof}
The statement in Theorem \eqref{eqinspace} is uniform in $\xi\in\pl\bbar{X}$, this is clear 
from the proof of that Theorem. If $\mc{F}$ is a fundamental domain of $X$  and 
$\eta$ is defined as in the proof of Theorem \eqref{eqinspace} in $\bbar{\mc{F}}\cap S^n$, then
the measure $dv_{\pl\bbar{X}}$ pulls back to $\bbar{\mc{F}}\cap S^n$ by $\pi_\Gamma$ to the measure
$\eta(\xi)^{-n}d\xi$ where $d\xi$ is the canonical measure on $S^n$. Since 
$E_X(n;\pi_{\gamma}(m),\pi_{\gamma}(\xi))=\eta(\xi)^n\sum_{\Gamma}E_0(n;\gamma m,\xi)$, we have \[\begin{split}
\int_{\bbar{\mc{F}}\cap S^n}E_X(n;\pi_{\gamma}(m),\pi_{\gamma}(\xi))\eta(\xi)^{-n}d\xi=&
\sum_{\gamma\in \Gamma}\int_{\bbar{\mc{F}}\cap S^n}E_0(n;m,\gamma^{-1}\xi)|D\gamma(\gamma^{-1}\xi)|^{-n}d\xi\\
=& \sum_{\gamma\in \Gamma}\int_{\gamma(\bbar{\mc{F}}\cap S^n)}E_0(n;m,\xi)d\xi\\
\int_{\pl\bbar{X}}E_X(n;m,\xi)dv_{\pl\bbar{X}}(\xi)=&\int_{\Omega_\Gamma}E_0(n;m,\xi)d\xi=\int_{S^n}E_0(n;m,\xi)d\xi={\rm vol}(S^n)
\end{split}\] 
where in the second line we used 
Lemma \ref{Dgamma}, in the third we used that the Jacobian of $\gamma$ at $\xi\in S^n$ 
is $|D\gamma(\xi)|^n$, the fourth line we used that the set of discontinuity $\Omega_\Gamma=\cup_{\gamma}\gamma(\bbar{\mc{F}}\cap S^n)$ has full measure in $S^n$ (for the standard measure on $S^n$).  This ends the proof.
\end{proof}

\subsection{Equidistribution in phase space}
In this section, we extend the previous equidistribution result to ``microlocal lifts" of the measures 
$$\Big|E_X(\ndemi+it;m,\xi)\Big|^2dv(m).$$
To this end, we recall a few basic facts on pseudo-differential operators, including semi-classical ones.
The proof of the micro-local extension will be based on semi-classical quantization and action on Lagrangian states.
\subsection{Pseudo-differential and semi-classical pseudo-differential operators}
Let $X$ be a complete manifold of dimension $d$. For $s\in\rr$, we define the space  $\Psi_c^{s}(X)$  of compactly supported pseudo-differential operators of order $s$  on $X$ to be the set of bounded operators 
$A: C^\infty(X)\to C_0^\infty(X)$, with Schwartz kernel compactly supported in $X\x X$, and which decomposes as a sum $A=A_{\rm sm}+A_{\rm sg}$
where $A_{\rm sm}$ has a smooth Schwartz kernel and $A_{\rm sg}$ 
can be written in any local chart $U$ with coordinate $z$ under the following form: 
for all $f\in C_0^\infty(X)$ supported in $U$, 
\[A_{\rm sg}f(z)=\frac{1}{(2\pi)^{d}}\int e^{i(z-z')\xi} a(z,\xi)f(z')dz'd\xi\]
with $a(z,\xi)$ is a {\it classical symbol} of order $s$ in $U$.  
A  classical symbol $a$ is an element in $C^\infty(U\x\rr^{d})$ such that 
for all $\alpha,\gamma,$ there exists $C>0$ so that
\[ |\pl_z^\gamma \pl_{\xi}^\alpha a(z,\xi)|\leq C(1+|\xi|)^{s-|\alpha|}, \quad 
a(z,\xi)\sim_{|\xi|\to \infty} \sum_{j=0}^\infty a_j(z,\xi)\]  
with $a_j$ some homogeneous functions of degree $s-j$ in $\xi$. The \emph{principal symbol} of $A$ is defined 
to be $a_0$, and it is invariantly defined as a homogeneous section of degree $s$ on $T^*X$.

Let $h\in(0,1)$ be a small parameter. For $s,k\in \rr$, we define the space $\Psi_{c}^{s,k}(X)$ of compactly supported semi-classical pseudo-differential operators of order $s$ to be the set of $h$-dependent bounded operators 
$A_h: C^\infty(X)\to C_0^\infty(X)$, with Schwartz kernel compactly supported in $X\x X$, and which decomposes as a sum $A_h=A_{h,{\rm sm}}+A_{h,{\rm sg}}$
where $A_{h,{\rm sm}}$ has a smooth Schwartz kernel satisfying $||A_{h,{\rm sm}}||_{L^2\to L^2}=O(h^\infty)$
and $A_{h,{\rm sg}}$ can be written in any local chart $U$ with coordinate $z$ under the following form: 
for all $f\in C_0^\infty(X)$ supported in $U$, 
\[A_{h,{\rm sg}}f(z)=\frac{1}{(2\pi h)^{d}}\int e^{i\frac{(z-z').\xi}{h}} a(h;z,\xi)f(z')dz'd\xi\]
with $a(h;z,\xi)$ a {\it semi-classical symbol} of order $s$ in $U$. A  semi-classical symbol $a(h; \cdot,\cdot)$ is an $h$-dependent element  in $C^\infty(U\x\rr^{d})$ such that 
for all $\alpha,\gamma$ and all $N$, there exists $C>0$  so that
\[\begin{gathered} 
|\pl_z^\gamma \pl_{\xi}^\alpha a(h;z,\xi)|\leq Ch^{-k}(1+|\xi|)^{s-|\alpha|}, \,\,
a(h;z,\xi)= \sum_{j=0}^{N-1} h^{-k+j}a_j(z,\xi)+h^Nr_{N}(h;z,\xi)
\end{gathered}\] 
where $a_j,r_N$ are smooth functions satisfying 
\[|\pl_z^\gamma \pl_{\xi}^\alpha a_j(z,\xi)|\leq C(1+|\xi|)^{s-|\alpha|},\quad 
|\pl_z^\gamma \pl_{\xi}^\alpha r_{N}(h;z,\xi)|\leq Ch^{-k}(1+|\xi|)^{s-N-|\alpha|}.\]
The space $\Psi_{c}^{0,0}(X)$ is an algebra under composition of operators, and the principal symbol of a 
composition is the product of their principal symbol. 
The \emph{semi-classical principal symbol} of $A_h$ is defined 
to be $h^{-k}a_0$, and it is invariantly defined as a section of $T^*X$. The \emph{microsupport} ${\rm WF}_h(A_h)$ of $A_h$ is the complement of the set of points $(z,\xi)\in T^*X$ such that 
$|\pl^\alpha a_h|=\mc{O}(h^\infty)$  near $(z,\xi)$, for all $\alpha$. 

\begin{lemma}\label{compos}
Let $B\in \Psi_c^0(X)$ and $A_h\in \Psi_c^{0,0}(X)$ with microsupport satisfying  
$${\rm dist}({\rm WF}_h(A_h),T^*_0X)>\eps,$$ 
for some $\eps>0$ uniform in  $h$, and where 
$T^*_0X$ is the zero section in $T^*X$.
Let $a_0,b_0$ be the semi-classical principal symbol and classical principal symbol of $A_h$ and $B$.
Then $BA_h\in \Psi_c^{0,0}(X)$ and the principal symbol of $BA_h$ is 
$b_0a_0$.  
\end{lemma}
\begin{proof} We decompose $B=B_{\rm sm}+B_{\rm sg}$, and using an adequate change of variables $B_{\rm sg}$ is given locally by
\[Bf(z)=\frac{1}{(2\pi h)^d}\int e^{i\frac{(z-z').\xi}{h}}b(z,\xi/h)f(z')d\xi dz'.\] 
Let $\chi\in C_0^\infty(\rr^d)$ so that its support does not intersect ${\rm WF}_h(A_h)$. Then 
\[B^1_{\rm sg}: f\mapsto B^1_{\rm sg}f(z)=\frac{1}{(2\pi h)^d}\int e^{i\frac{(z-z').\xi}{h}}(1-\chi(\xi))b(z,\xi/h)f(z')d\xi dz'\]
is an operator in $\Psi_c^{0,0}(X)$ with semi-classical symbol $(1-\chi(\xi))b_0(z,\xi)$. On the other hand, by using non-stationary phase and the fact that ${\rm supp}(\chi)\cap {\rm WF}_{h}(A_h)=\emptyset$, we deduce that   
$(B_{\rm sg}-B^1_{\rm sg})A_h$ has a smooth kernel and is an 
$\mc{O}(h^\infty)$ as an operator on $L^2(X)$. The composition $B_{\rm sm}A_h$ 
has the same properties by applying again non-stationary phase (using that $B_{\rm sm}$ has kernel supported away from the diagonal and $a(h;z,\xi)=0$ near $\xi=0$). We conclude that $BA_h\in \Psi_c^{0,0}(X)$ and its principal symbol is $(1-\chi)b_0a_0=b_0a_0$.
\end{proof}

\subsection{Action of semi-classical operators on semi-classical Lagrangian distributions}
Let us consider a Lagrangian submanifold 
$\mc{L}\subset T^*X$ which is the graph of $d\varphi$ for some smooth function $\varphi$ on $X$
\[\mc{L}=\{ (m,d\varphi(m))\in T^*X; m\in X\}.\]
A (compactly supported) \emph{semi-classical Lagrangian distribution} with Lagrangian $\mc{L}$ is a function on $X$ which can be written, for all $N\in \nn$, under the form 
\[u(m)=e^{i\frac{\varphi(m)}{h}}(\sum_{j=1}^{N-1}a_j(m)h^j+h^{N}r_{N,h}(m)) \] 
for some $a_j,r_N\in C_0^\infty(X)$ and $|\pl_m^\alpha r_{N,h}|=\mc{O}(1)$ for all $\alpha$.
We call $a_0$ the principal symbol of $u$.
Semi-classical pseudo-differential operators preserve the space of lagrangian distributions with Lagragian $\mc{L}$,
as is proved for example in  \cite[Lemma 4.1]{NonZw}, which we recall here
\begin{lemma}\label{nonzwo}
Let $T\in \Psi^{0,0}_c(X)$ be a compactly supported semi-classical operator with principal 
symbol $\alpha_0$ and $u$ be a compactly supported semi-classical Lagrangian distribution with Lagrangian $\mc{L}$ and principal symbol $a_0$.
Then $Tu$ is a compactly supported Lagrangian distribution with Lagrangian $\mc{L}$ and with principal symbol
$b_0(m):=\alpha_0(m,d\varphi(m))a_0(m)$. More precisely
\[Tu(m)=e^{i\frac{\varphi(m)}{h}}(b_0(m)+hr_{1,h}(m))
\]
with $r_{1,h}\in C_0^\infty(X)$ satisfying $||r_{1,h}||_{L^2}\leq C( ||u||_{L^2}+||a_0||_{C^{n+3}})$ for some $C>0$ depending only on semi norms of the local symbol of $T$ and on semi-norms of $\varphi$.
\end{lemma}

We will use the semi-classical notations in this section: let $\la_h=\ndemi+\frac{i}{h}$, then for $\xi\in S^n$ fixed
\begin{equation}\label{phi}
E_0(\la_h; m,\xi)=E_0(\ndemi;m,\xi)e^{i\frac{\phi_\xi (m)}{h}} \, \textrm{ with }\phi_\xi(m):=\log\Big(\frac{1-|m|^2}{4|m-\xi|^2}\Big).
\end{equation}
As a function of $m$, $E_0(\la_h;m,\xi)$ is a semi-classical Lagrangian distribution associated to the Lagrangian 
submanifold  
\[\mc{L}_\xi:= \{ (m,d\phi_\xi(m))\in T_m^*\hh^{n+1}; m\in \hh^{n+1}\}\subset T^*\hh^{n+1}.\]

Let $\Gamma$ be a convex co-compact group of isometries of $\hh^{n+1}$. 
For a given $\xi\in \Omega_\Gamma$, let us define the (closure of) discrete union of Lagrangians 
\begin{equation}\label{mcL}  
\mc{L}^\Gamma_\xi:=\overline{\bigcup_{\gamma\in \Gamma} \mc{L}_{\gamma \xi}}.
\end{equation}
If $\pi_\Gamma:\hh^{n+1}\to \Gamma\backslash \hh^{n+1}=X$ is the projection to the quotient, then
$$\pi_\Gamma(\mc{L}^\Gamma_\xi)=\overline{\bigcup_{\gamma\in \Gamma}\pi_\Gamma(\mc{L}_{\gamma\xi})}$$ 
is the closure of a countable superposition of Lagrangian submanifolds of $T^*X$. Over a point $m\in X$, 
$\pi_\Gamma(\mc{L}^\Gamma)\cap S_m^*X$ corresponds to the (closure of) set of directions $v(m)\in S_m^*X$ 
such that the geodesic starting at $(m,v(m))$ tends to $\pi_{\Gamma}(\xi)\in \Gamma\backslash \Omega_\Gamma=\pl\bbar{X}$.
We point out that since $\Gamma$ is discrete, for all $\xi \in \Omega_\Gamma$, we have
$$\overline{\Gamma . \xi}=\left ( \bigcup_{\gamma \in \Gamma}\gamma \xi \right)\sqcup \Lambda_\Gamma, $$
see Beardon \cite{Bear}, Theorem 5.3.9. In other words, the set of accumulation points of the orbit $\Gamma . \xi$ is the full limit set 
$\Lambda_\Gamma$. Since $\mc{L}^\Gamma_\xi$ is locally diffeomorphic to $\hh^{n+1}\times \overline{\Gamma . \xi}$,
some basic facts of dimension theory show that the Hausdorff dimension of $\mc{L}^\Gamma_\xi$ is $n+1+\delta_\Gamma$. 
We show the following 
\begin{theorem}\label{phasesp}
Let $X=\Gamma\backslash \hh^{n+1}$ be a convex co-compact hyperbolic manifold, and assume 
$\delta_\Gamma<\ndemi$.
Let $A\in \Psi^{0}_c(X)$ be a compactly supported classical pseudo-differential operator of order 
$0$ with principal symbol $a\in C_0^\infty(X,T^*X)$. Then for $\xi\in \pl\bbar{X}=\Gamma\backslash \Omega_\Gamma$, we have as $h\to 0$ and for if $\la_h:=n/2+i/h$
\[\cjg AE(\la_h;\cdot,\xi),E(\la_h;\cdot,\xi)\cjd_{L^2(X)}=\int_{S^*X} a\,d\mu_{\xi}+\mc{O}(h^{\min(1,n-2\delta)})\]
where $d\mu_{\xi}$ is a measure supported on $\pi_{\Gamma}(\mc{L}^\Gamma_\xi)$ as defined in \eqref{mcL}.
More precisely, denoting also $a$ and $\xi$ for their lifts to respectively $S^*\hh^{n+1}$ and $\Omega_{\Gamma}$ through  $\pi_\Gamma$, and if $\mc{F}$ is a fundamental domain of $\Gamma$,
$\int_{S^*X}ad\mu_\xi$ can be written as a converging series
\begin{equation}\label{converging}
\int_{S^*X} a(m,v)d\mu_{\xi}(m,v):=\int_{\mc{F}} \sum_{\gamma\in \Gamma}a(m,d\phi_{\gamma \xi}(m))
E_0(n;m,\gamma\xi)|D\gamma(\xi)|^{n}dv(m).
\end{equation}
\end{theorem}
\begin{proof}
First we study $AE_X(\la_h;\cdot,\xi)$. Let $\chi,\chi'\in C_0^\infty(X)$ be functions such that $\chi A\chi=A$ and $\chi'\chi=\chi$, and let $\psi_h\in \Psi_c^{0,0}(X)$ be a semi-classical pseudo-differential operator  
with ${\rm WF}_h(\psi_h)\subset \{(m,v)\in T^*M; ||v|-1|\leq 1/2\}$ and 
with semi-classical principal symbol $\psi$ equal to $1$ on $ \{(m,v)\in T^*M; ||v|-1|\leq 1/4\}$. 
We first claim that 
$\chi'(1-\psi_h)\chi E_X(\la_h;\cdot,\xi)=\mc{O}(h)$ in $L^2$. Indeed,
let $P_h:=h^2(\Delta_X-n^2/4)-1$, then using that $P_hE_X(\la_h;\cdot,\xi)=0$, we deduce that
\[ P_h\chi'(1-\psi_h)\chi E_X(\la_h;\cdot,\xi)=[P_h,\chi'(1-\psi_h)\chi] E_X(\la_h;\cdot,\xi)\] 
which has $L^2$ norm $\mc{O}(h)$ since the commutator is in $\Psi^{0,-1}_c(X)$ (the semi-classical 
symbol of $\psi_h$ being compactly supported) and $||\chi E_X(\la_h,\cdot,\xi)||_{L^2}=\mc{O}(1)$. The claim
follows by the invertibility of the semi-classical principal symbol of $P_h(1-\psi_h)$ on the microsupport of 
$\chi E_X(\la_h;\cdot,\xi)$. 
Therefore 
\begin{equation}\label{AEX}
AE_X(\la_h;\cdot,\xi)=A\psi_h\chi E_X(\la_h;\cdot,\xi)+\mc{O}_{L^2}(h).
\end{equation}
By Lemma \ref{compos}, $A_h:=A\psi_h\chi'\in \Psi_c^{0,0}(X)$ is a semi-classical operator with 
semi-classical principal symbol $a(m,v)\psi(m,v)$.

Now we lift all objects to the covering $\hh^{n+1}$ and we can assume without loss of generality that 
 $A$ has kernel compactly supported in $\mc{F}\x\mc{F}$ where $\mc{F}$ is a fundamental domain of $\Gamma$.
Recall that $E_X(\la_h;m,\xi)$ can be expressed as a converging series 
\[\begin{split}
E_X(\la_h;m,\xi)&=\sum_{\gamma\in \Gamma} E_0(\ndemi;m,\gamma\xi)
|D\gamma^{-1}(\gamma\xi)|^{-\la_h}
e^{i\phi_{\gamma\xi}(m)/h}\\
&=\sum_{\gamma\in \Gamma} E_0(\ndemi;m,\gamma\xi)
|D\gamma(\xi)|^{\la_h}
e^{i\phi_{\gamma\xi}(m)/h}
\end{split}\]
For the second line, we have used that $\gamma$ is a M\"obius transform, so that  
$$|D\gamma^{-1}(\gamma\xi)|^{-1}=|D\gamma(\xi)|,$$ 
series being convergent when $\delta_\Gamma<n/2$ 
since $|D\gamma(\xi)|\leq Ce^{-d(o,\gamma o)}$ uniformly for $\xi\in \bbar{F}\cap S^n$, by \eqref{formulebeardon}. 
In particular, one has $||\chi E_X(\la_h;\cdot,\xi)||_{L^2}=\mc{O}(1)$, uniformly in $\xi\in \pl\bbar{X}$.

The action of $A_h$ on each term of the series, ie. $A_h(\chi e^{i\phi_{\gamma\xi}/h}E_0(\ndemi;\cdot,\gamma\xi))$, is described by  Lemma \ref{nonzwo}, $A_h(\chi e^{i\phi_{\gamma\xi}/h}E_0(\ndemi;\cdot,\gamma\xi))$ is a Lagrangian distribution on 
$\hh^{n+1}$ associated to the Lagrangian $\mc{L}_{\gamma\xi}$, and compactly supported in $\mc{F}$:
\begin{equation}\begin{split}\label{evol}
A_h(\chi e^{i\phi_{\gamma\xi}/h}E_0(\ndemi;\cdot,\gamma\xi))=e^{i\phi_{\gamma\xi}/h}\Big( &
\chi(m)E_0(\ndemi;m,\gamma\xi)a(m,d\phi_{\gamma\xi}(m))\psi(m,d\phi_{\gamma\xi}(m))\\
& +hr_h(m,\gamma\xi) \Big)
\end{split}\end{equation}
with $||r_h(\cdot,\gamma\xi)||_{L^2}=\mc{O}(1)$ bounded by a constant times finitely many semi-norms of 
$a$, $E_0(\ndemi,\cdot,\xi)$ and of the phase 
$\phi_{\gamma\xi}$. But it is easily seen that for all $\alpha\in \nn^{n+1}$ there exists $C_\alpha>0$ such that
for all $\gamma\in \Gamma$ and $\xi\in \bbar{\mc{F}}\cap S^n$, 
\begin{equation}\label{derivephi} 
\sup_{m\in\supp(\chi)}|\pl_m^\alpha \phi_{\gamma\xi}(m)| \leq C_\alpha,  \quad 
\sup_{m\in\supp(\chi)}|\pl_m^\alpha E_0(\ndemi; m,\gamma\xi)|\leq C_\alpha
\end{equation}
therefore $||r_h(\cdot,\gamma\xi)||_{L^2}$ is uniformly bounded with respect to $h,\gamma$ and $\xi$.
Notice, by \eqref{dphi} and since $\psi=1$ on $\{(m,v)\in T^*X; ||v|-1|<1/4\}$, we have $\psi(m,d\phi_{\gamma\xi}(m))=1$.
We thus deduce that there exists $C>0$ uniform in $h,\gamma,\xi$ such that 
\[\Big|\Big|e^{i\phi_{\gamma\xi}/h}\Big(
\chi E_0(\ndemi;\cdot,\gamma\xi)a(\cdot,d\phi_{\gamma\xi}(\cdot))+
hr_h(\cdot,\gamma\xi) \Big)\Big|\Big|_{L^2}\leq C .\]
 Consequently,  using again $\delta_{\Gamma}<n/2$ with $|D\gamma(\xi)|\leq Ce^{-d(o,\gamma o)}$,
\[ \sum_{\gamma\in \Gamma} \Big|\Big|e^{i\phi_{\gamma\xi}/h}|D\gamma(\xi)|^{\la_h}|\Big(
\chi E_0(\ndemi;\cdot,\gamma\xi)a(\cdot,d\phi_{\gamma\xi}(\cdot))+
hr_h(\cdot,\gamma\xi) \Big)\Big|\Big|_{L^2} <\infty \]
and we deduce that 
\[\begin{split}
A_h\chi E_X(\la_h;m,\xi)=&\sum_{\gamma\in \Gamma}A_h(\chi e^{i\phi_{\gamma\xi}/h}E_0(\ndemi;\cdot,\gamma\xi))(m)
|D\gamma(\xi)|^{\la_h}\\
=& \sum_{\gamma\in \Gamma}e^{i\frac{\phi_{\gamma\xi(m)}}{h}} |D\gamma(\xi)|^{\la_h}\Big(
E_0(\ndemi;m,\gamma\xi)a(m,d\phi_{\gamma\xi}(m))+hr_h(m,\gamma\xi) \Big).
\end{split}\]
Integrating this against $\bbar{E_X(\la_h,\cdot,\xi)}=E_X(\bbar{\la_h},\cdot,\xi)$, we have 
\begin{equation}\label{presqfini}\begin{gathered}
\cjg A_h\chi E_X(\la_h;m,\xi),E_X(\la_h,\cdot,\xi)\cjd\\
= \int_{\mc{F}} 
\sum_{\gamma\in\Gamma}|D\gamma(\xi)|^{n}E_0(n;m,\gamma\xi)a(m,d\phi_{\gamma\xi})dv(m)+\mc{O}(h)
\\
 +\int_{\mc{F}} \sum_{\gamma\not=\gamma'}e^{i\frac{B_{\xi}(\gamma m,\gamma' m)}{h}}|D\gamma(\xi)|^{\la_h}|D\gamma'(\xi)|^{\bbar{\la_h}}
\Big(E_0(1;m,\gamma\xi)E_0(1;m,\gamma'\xi)\Big)^\ndemi a(m,d\phi_{\gamma\xi}(m))dv(m)
\end{gathered}
\end{equation}
where $B_\xi(m,m')$ is the Buseman function of \eqref{buseman} and the $\mc{O}(h)$ term comes from the 
remainder $h\cjg\sum_{\gamma\in\Gamma}e^{i\phi_{\gamma\xi}/h}r_h(\cdot,\gamma\xi)|D\gamma(\xi)|^{\la_h},E_X(\la_h,\cdot,\xi)\cjd$
which is uniformly bounded in $\xi\in \pl\bbar{X}$ by the discussion above. Now it remains to apply 
Lemma \ref{Idecays} to deal with the off-diagonal term in the second line of \eqref{presqfini}, this gives that this term is 
a $\mc{O}(h^{n-2\delta_\Gamma})$. Notice that the amplitude $a(m,d\phi_{\gamma\xi}(m))$ 
here depends on $\gamma$, but this does not affect the argument of Lemma \ref{Idecays}, since the derivatives 
of the amplitude are uniformly bounded in $\gamma,\xi$ by \eqref{derivephi}. The proof of the Theorem 
is then finished by combining this with \eqref{AEX}.
\end{proof}

As a corollary, we also get 
\begin{cor}\label{inaverage}
Let $X=\Gamma\backslash \hh^{n+1}$ be a convex co-compact hyperbolic manifold with
$\delta_\Gamma<\ndemi$.  Let $E_X(\la;\cdot,\xi)$ and $dv_{\pl\bbar{X}}$ be the Eisenstein series and the 
Riemannian measure on $\pl\bbar{X}$ defined with a boundary defining function $x$ as before.
If $A\in \Psi^{0}_c(X)$ is a compactly supported classical pseudo-differential operator of order 
$0$ with principal symbol $a\in C_0^\infty(X,T^*X)$, then setting $\la_h=n/2+i/h$, we have as $h\to 0$
\[ \int_{\pl \bbar{X}}\cjg AE_X(\la_h;\cdot,\xi),E_X(\la_h;\cdot,\xi)\cjd_{L^2(X)} dv_{\pl\bbar{X}}(\xi)=\int_{S^*X} a(m,v) d\mu(m,v)+ \mc{O}(h^{\min(n-2\delta,1)})
\]
where $S^*X$ is the unit cotangent bundle and $d\mu$ the Liouville measure. 
\end{cor}
\begin{proof}
We integrate \eqref{converging} in the $\xi\in\pl\bbar{X}$ variable with respect to $dv_{\pl\bbar{X}}$ like in the proof 
of Corollary \ref{firstcor} by lifting to the covering: we have, modulo $\mc{O}(h^{\min(1,n-2\delta)})$.
\[\begin{gathered}
\int_{\pl \bbar{X}}\cjg AE_X(\la_h;\cdot,\xi),E_X(\la_h;\cdot,\xi)\cjd_{L^2(X)} dv_{\pl\bbar{X}}(\xi)=\\
\int_{\mc{F}} \sum_{\gamma\in \Gamma}\int_{\xi\in \bbar{\mc{F}}\cap S^n}a(m,d\phi_{\gamma \xi}(m))
E_0(n;m,\gamma\xi)|D\gamma(\xi)|^{n}dv(m)d\xi=\\
\int_{\mc{F}} \int_{\xi\in \Omega_{\Gamma}}a(m,d\phi_{\xi}(m))
E_0(n;m,\xi)dv(m)d\xi
\end{gathered}\]
and the last term is exactly $\int_{S^*X} a(m,v) d\mu(m,v)$ by using \eqref{gammanu} and \eqref{psimu}.
\end{proof}

\section{Gutzwiller type formula}
In this section, we shall give a more precise description of the equidistribution in space when the Eisenstein series 
are averaged over the boundary. Through the Stone identity \eqref{stone}, this reduces to a sort of Gutzwiller trace formula. We prove the following 
\begin{theorem}
Let $X=\Gamma\backslash \hh^{n+1}$ be a convex co-compact hyperbolic manifold and assume that the limit set of 
$\Gamma$ has Hausdorff dimension $\delta_\Gamma<n/2$.
Let  $a\in C_0^\infty(X)$ and $\mc{C}$ be the set of all closed geodesics, 
then for all $N\in \nn$, one has the expansion as $t\to \infty$,
\begin{equation}
\begin{gathered}
\int_{\pl \bbar{X}}\int_X a(m)|E(\ndemi+it;m,\xi)|^2dm\,d\xi_{\pl \bar{X}}={\rm vol}(S^n)
\int_X a \, dv\\
+ L(t)\sum_{\gamma\in \mc{C}}
e^{-(\ndemi+it)\ell(\gamma)}\sum_{k=0}^{N}\frac{H_{k}(\ell(\gamma))}
{|\det({\rm Id}-e^{-\ell(\gamma)}R_\gamma^{-1})|}t^{-k}
\int_{\gamma} P_ka\, d\mu +\\
+ L(-t)\sum_{\gamma\in \mc{C}}
e^{-(\ndemi-it)\ell(\gamma)}\sum_{k=0}^{N}\frac{H_{k}(\ell(\gamma))}{|\det({\rm Id}-e^{-\ell(\gamma)}R_\gamma^{-1})|}(-t)^{-k}
\int_{\gamma} P_ka\, d\mu + \mc{O}(t^{-N-n-1}).
\end{gathered}
\end{equation}
where $H_k(\ell(\gamma))$ is an explicit bounded function of $\ell(\gamma)$, $P_k$ are differential operators of order 
$2k$ with coefficients uniformly bounded in terms of $\gamma\in\mc{C}$, $d\mu$ is the Riemannian measure induced on $\gamma$, $R_\gamma\in {\rm SO}(n)$ is the holonomy along $\gamma$ 
and 
\[\begin{gathered}
L(t)=t^{-(n-1)/2}\frac{2^{\ndemi+3-2it}|\Gamma(it)|^2}{\Gamma(it+\demi)\Gamma(\ndemi-it)}=\mc{O}(t^{-n})\\
 P_0= 1, \quad  H_{0}(\ell(\gamma))=2^\ndemi\pi^{\frac{n-1}{2}}e^{i\frac{n-1}{4}\pi}
\end{gathered}\] 
\end{theorem}
\begin{proof}
We write $\mu_t(a)$ for the left hand side of \eqref{qer}. We will assume, without losing generality, 
that the support of $a$ is small enough so that there exists a fundamental domain 
$\mc{F}\subset \hh^{n+1}$ with $\pl\mc{F}\cap {\rm supp}(a)=\emptyset$.
 We have by \eqref{stone}
\[\begin{split}
\mu_t(a)=&\int_{X}\int_{\pl X} a(m)E_X(\ndemi+it;m,y)E_X(\ndemi-it;m,y) dv(m) dv_{\pl\bbar{X}}(y)\\
&= \frac{4\pi t}{|C(\ndemi+it)|^2} \int_X a(m) d\Pi_X(t;m,m) dv(m)\\
\mu_t(a) &=\frac{4\pi t}{|C(\ndemi+it)|^2} \sum_{\gamma\in \Gamma} \int_{\mc{F}} a(m)d\Pi_0(t;\gamma m,m) dv(m).
\end{split} \]
where $d\Pi_X$ is the spectral measure of $\Delta$ on $X$ and $d\Pi_0$ the spectral measure of $\Delta$ on $\hh^{n+1}$.
For the $\gamma={\rm Id}$ term, the spectral measure on $\hh^{n+1}$ restricted on the diagonal is a constant function  given by
\[ \alpha(t):=4\pi t  \, d\Pi_0(t;m,m)=\pi^{-\ndemi}\frac{\Gamma(n/2)}{\Gamma(n)}\frac{\Gamma(\ndemi+it)\Gamma(\ndemi-it)}{\Gamma(it)\Gamma(-it)}=|C(\ndemi+it)|^2{\rm vol}(S^n).\]
The terms with $\gamma\not={\rm Id}$ are actually lower order as $t\to \infty$: we claim that 
\begin{equation}\label{decay}
\sum_{\gamma\in \Gamma\setminus{\rm Id}} \int_{\mc{F}} a(m)td\Pi_0(t;\gamma m,m) dv(m)=\mc{O}(1)
\end{equation}
and we will actually obtain a much better description.

From the explicit formula of the resolvent for the Laplacian on hyperbolic space in \eqref{R0}, 
we have that 
\[4\pi td\Pi_0(t;\gamma m,m)=-4t {\rm Im}(F_{t}(\sigma(d(m,\gamma m))))\] 
where $\sigma(r)=1/\cosh(r)$, $d(\cdot,\cdot)$ denotes the hyperbolic distance and 
\[\begin{gathered}
F_t(\sigma):=M(t)\sigma^{\ndemi+it}\int_{0}^1 (u(1-u))^{it-\demi}(\sigma(1-2u)+1)^{-\ndemi-it}du\\
M(t):=\pi^{-\frac{n+1}{2}}2^{1-\ndemi-it}\frac{\Gamma(\ndemi+it)}{\Gamma(\demi+it)}=\mc{O}(t^{\frac{n-1}{2}})
\end{gathered}\]
\begin{lemma}\label{expansion}
Let $\eps>0$, and $\sigma(r)=\cosh(r)^{-1}$, then there exist some smooth functions $c_k\in C^\infty([0,1-\eps))$ such that for all $N$, we have
as $t\to \infty$ 
\[
F_t(\sigma(r))= M(t)2^{-it}e^{-itr}\sigma(r)^{\ndemi}\Big(\sum_{k=0}^N t^{-\demi-k}c_k(\sigma(r)) + t^{-\demi-N-1}R_N(\sigma(r),t)\Big).
\]
with the remainder $R_N$ satisfying $|\pl_\sigma^\ell R_N(\sigma,t)|\leq C_\ell$, for all $\ell\in\nn_0$, 
$\sigma\in[0,1-\eps)$, $t>t_0>0$, and $c_0(\sigma)=\pi^{\demi}e^{-i\pi/4}(1-\sigma^2)^{-n/4}$. 
\end{lemma}
\begin{proof}
Assuming that $\sigma<1-\eps$, then in the integral defining $F_t(u)$, one can use the stationary phase method
(e.g. \cite[Th 3.10]{EvZw} ) to get an asymptotic expansion in $t\to \infty$:
the phase has a unique nondegenerate critical point in $(0,1)$ given by $u_\sigma=\frac{1}{2\sigma}(1+\sigma-\sqrt{1-\sigma^2})$, the Hessian is a smooth function of $\sigma\in [0,1-\eps)$ which is bounded from 
below by a positive constant depending only on $\eps>0$.
Using a cutoff function near the point $u_\sigma$ 
and using integration by parts for the non-stationary part, one deduces as $t\to \infty$
\[\sigma^{\ndemi+it}\int_{0}^1 (u(1-u))^{it-\demi}(\sigma(1-2u)+1)^{-\ndemi-it}du\sim \sigma^{\ndemi}\sum_{k=0}^\infty t^{-\demi-k}c_k(\sigma) 
\Big(\frac{1-\sqrt{1-\sigma^2}}{2\sigma}\Big)^{it}\]
with $c_k(\sigma)$ some smooth function of $\sigma\in[0,1-\eps]$ and a straightforward computation gives
\begin{equation}\label{c_0}
c_0(\sigma)= \pi^{\demi}e^{-i\pi/4}(1-\sigma^2)^{-n/4}.
\end{equation} 
Since 
$\frac{1-\sqrt{1-\sigma(r)^2}}{2\sigma(r)}=\demi e^{-r}$, we have an expansion, for all $N\in\nn$, as $t\to \infty$
\begin{equation}\label{stat}
F_t(\sigma(r))= M(t)2^{-it}e^{-itr}\sigma(r)^{\ndemi}\Big(\sum_{k=0}^N t^{-\demi-k}c_k(\sigma(r)) + \mc{O}_\sigma(t^{-\demi-N-1})\Big).
\end{equation}
with a remainder uniform in $\sigma<1-\eps$, as well as  its derivatives with respect to $\sigma$.
\end{proof}
Notice that here $d(m,\gamma m)> C_0$ 
for some $C_0>0$ uniform in $m$ (by the convex co-compactness of $\Gamma$), then $\sigma(d(m,\gamma m))<1-\eps$ for some uniform $\eps$ and thus $F_t(\sigma(m,\gamma m))$ is smooth in $m\in K$ for any fixed compact set $K\subset \hh^{n+1}$.

The displacement function $r(m,\gamma m)$, or more precisely $\sinh(\demi r(m,\gamma m))$, is easily 
computed in dimension $n+1=2$, it is given by $\sinh(\demi r(m,\gamma m))=\cosh(d(m,{\rm axis}(\gamma)))\sinh(\demi\ell(\gamma))$. In higher dimension it is more complicated. 
Consider the half plane model $\{x>0,y\in \rr^n\}$ of $\hh^{n+1}$. We can assume that $\infty$ is not in the limit set of $\Gamma$,
so that a fundamental domain can be taken in a ball $x^2+|y|^2\leq R$ for some $R$.
For all $\gamma\in \Gamma$, there is an isometry 
$h_\gamma$ such that
\begin{equation}\label{tgamma}
h_\gamma\circ\gamma\circ h_\gamma^{-1}(x,y)=e^{\ell(\gamma)}(x,R_\gamma y)
\end{equation}
where $\ell(\gamma)>0$ is the translation length of $\gamma$
and $R_\gamma\in {\rm SO}(n)$ the holonomy associated to $\gamma$.
Since $\sigma(d(m,\gamma m))=\sigma(d(h_\gamma m,h_\gamma\circ\gamma\circ h_{\gamma}^{-1}h_\gamma m))$, one has  
\begin{equation}\label{sigmad}
\sigma(d(h_\gamma^{-1}m,\gamma \circ h_\gamma^{-1}m))=e^{-\ell(\gamma)}\frac{2x^2}{x^2(1+e^{-2\ell(\gamma)})+|(e^{-\ell(\gamma)}{\rm Id}-R_\gamma) y|^2}\end{equation}
where $m=(x,y)$. 
A consequence of this and Lemma \ref{expansion} is 

\begin{cor}
With $h_\gamma$ defined by \eqref{tgamma}, we have  as $t\to\infty$
\begin{equation}\label{integrale}
\begin{split}
\int_{\mc{F}} a(m)F_{t}(\sigma(d(m,\gamma m))) dv(m)
=&\frac{M(t)}{2^{it}}\sum_{k=0}^N t^{-k-\demi}\int_{h_\gamma(F)}e^{-it r_\gamma(m)}b_k(\sigma_\gamma(m))
a(h^{-1}_\gamma(m))\frac{dm}{x^{n+1}}\\
&+\mc{O}(\max_{m\in\supp(a)}e^{-\ndemi d(m,\alpha m)}t^{-N+n-2})
\end{split}
\end{equation}
where 
$r_\gamma(m)=:d(h_\gamma^{-1}m, \gamma  \circ h_\gamma^{-1}m)$, $
\sigma_\gamma(m):=\sigma(d(h_\gamma^{-1}m, \gamma \circ  h_\gamma^{-1}m))$ and 
$b_k(\sigma):=\sigma^{\ndemi}c_k(\sigma)$.
\end{cor}
We now want to use stationary phase in the space variable $m\in\mc{F}$ in the integral \eqref{integrale}.
Let us define the following function
$\Phi_\gamma(m):=\log(\sigma(d(h_\gamma^{-1}m, \gamma \circ  h_\gamma^{-1}m))$, so that 
\begin{equation}\label{drgamma} 
dr_\gamma(m)=\frac{d\Phi_\gamma(m)}{\tanh(r_\gamma(m))}.
\end{equation}
We first calculate the critical points of the phase $r_\gamma(m)$, which in turn are the critical points of the function
$\Phi_\gamma(m)$ in the $m=(x,y)$ coordinates of $\hh^{n+1}$
\begin{equation}\label{dPhi}
\begin{split}
d\Phi_\gamma(x,y)= & \Big(\frac{2|(e^{-\ell(\gamma)}{\rm Id}-R_\gamma)y|^2}{x^2(1+e^{-2\ell(\gamma)})+|(e^{-\ell(\gamma)}{\rm Id}-R_\gamma)y|^2}\Big)\frac{dx}{x}\\
&-x\Big( \frac{2(e^{-\ell(\gamma)}{\rm Id}-R_\gamma)^T(e^{-\ell(\gamma)}{\rm Id}-R_\gamma)y}{x^2(1+e^{-2\ell(\gamma)})+|(e^{-\ell(\gamma)}{\rm Id}-R_\gamma)y|^2}\Big).\frac{dy}{x}
\end{split}.\end{equation}
Therefore the critical points are located on the line $L=\{y=0\}$ which is  the image of the axis of $\gamma$ by $h_\gamma$.
For each slice $x={\rm cst}$,  we will integrate 
in the $y$ variable in \eqref{integrale} by using the stationary phase, we then first
need to check that $\det({\rm Hess}(\Phi_\gamma(x,0)))\not= 0$ for all $x$, where the Hessian is with respect 
to the $y\in\rr^n$ variable. The Hessian (in y variable) at $y=0$ is given by 
\[{\rm Hess}(\Phi_\gamma(x,0))=-2\frac{(e^{-\ell(\gamma)}{\rm Id}-R_\gamma)^T(e^{-\ell(\gamma)}{\rm Id}-R_\gamma)}{x^2(1+e^{-2\ell(\gamma)})}\]
and thus, using that $r_\gamma(x,0)=\ell(\gamma)$,
\begin{equation}\label{hessien}
{\rm Hess}(-r_\gamma(x,0))=\frac{2}{x^2(1-e^{-2\ell(\gamma)})}(e^{-\ell(\gamma)}{\rm Id}-R_\gamma)^T(e^{-\ell(\gamma)}{\rm Id}-R_\gamma).
\end{equation} 
with $|\det(e^{-\ell(\gamma)}{\rm Id}-R_\gamma)|^2=|\det({\rm Id}-e^{-\ell(\gamma)}R_\gamma^{-1})|^2>0$.

Therefore the phase is non-stationary 
if the axis of $\gamma$ does not intersect the support of $a(m)$, contained in the fundamental domain $\mc{F}$, and the integral is then 
a $\mc{O}(t^{-\infty})$ in this case (constants are depending on $\gamma$).  We shall split the sum over $\gamma\in \Gamma$ into 
conjugacy classes $[\gamma]$ in the group, as is usual in Selberg type analysis. 
For each primitive element $\gamma$, the conjugacy 
class of $\gamma$ corresponds to a primitive geodesic on $X$, which is given by 
\[ \pi_\Gamma\Big(\bigcup_{\beta \gamma \beta^{-1}\in [\gamma]}\beta({\rm axis}_\gamma)\cap \mc{F}\Big) \]
where $\pi_\Gamma: \hh^{n+1}\to \Gamma\backslash \hh^{n+1}$ is the natural projection, $[\gamma]$ denotes the conjugacy class of 
$\gamma$ and ${\rm axis}_\gamma$ is the axis of  $\gamma$. Moreover, there are only finitely many geodesics $\beta({\rm axis}_\gamma)$ that 
intersect $\mc{F}$, we then partition $[\gamma]$ into $A_1(\gamma)\cup A_2(\gamma)$ so that
 $\alpha=\beta \gamma \beta^{-1}\in A_2(\gamma)$ iff $d(\beta({\rm axis}_\gamma),\mc{F})>1$. If $\gamma$ is not primitive, one can do the same thing with $\gamma_0^k$ for some primitive $\gamma_0$. 

We first show the 
\begin{lemma}\label{A_2}
Let $\gamma\in \Gamma$ be a primitive element, then that for all $N>0$ there is $C_N$ depending only on $N$ such that for all $\alpha\in A_2(\gamma)$
\begin{equation}\label{A2}
\Big|\int_\mc{F} a(m)F_{t}(\sigma(d(m,\alpha m))) dv(m)\Big|\leq 
C_Nt^{-N+(n-1)/2}\max_{m\in \supp(a)} e^{-\ndemi d(m,\alpha m)}.
\end{equation}
\end{lemma} 
 \begin{proof}
For $\alpha\in A_2(\gamma)$, we then use non-stationary phase to evaluate \eqref{integrale}, i.e. integration by parts: 
 \begin{equation}\label{intbyparts}
 \begin{gathered}
\int_{h_\alpha(\mc{F})}e^{-it r_\alpha(m)}b_k(\sigma_\alpha(m))
a(h^{-1}_\alpha(m))\frac{dxdy}{x^{n+1}}\\
=-\frac{1}{it}\int_{h_\alpha (\mc{F})} \frac{\pl_x(e^{-itr_{\alpha}(m)})}
{\pl_x\Phi_\alpha(m)}\tanh(r_\alpha(m))b_k(\sigma_\alpha(m))
a(h^{-1}_\alpha (m))\frac{dxdy}{x^{n+1}}\\
=\frac{1}{it}\int_{h_\alpha (\mc{F})} e^{-itr_{\alpha}(m)}(x\pl_x-n)\Big(
\frac{\tanh(r_\alpha(m))b_k(\sigma_\alpha(m))a(h^{-1}_\alpha m)}{{x\pl_x\Phi_\alpha(m)}}\Big)
\frac{dxdy}{x^{n+1}}\\
=\frac{1}{(it)^N}\int_{h_\alpha(\mc{F})} e^{-itr_{\alpha}(m)}\Big((x\pl_x-n)\frac{\tanh(r_\alpha(m))}
{x\pl_x\Phi_\alpha}\Big)^N\Big(
b_k(\sigma_\alpha(m))a(h^{-1}_\alpha m)\Big)
\frac{dxdy}{x^{n+1}}
 \end{gathered}
 \end{equation}
where $N\in\nn$ will be chosen later. 
To estimate the terms in the last integral,
we need an expression for $h_\alpha$: if $p^1_{\alpha},p_{\alpha}^2\in \rr^n$ are the two fixed points of $\alpha$, define $p_\alpha:=p^2_{\alpha}-p^1_{\alpha}$, then $h_\alpha$ will be the composition of 
the translation $t_{\alpha}: m\to m-p^1_{\alpha}$ with  the reflection in the sphere 
$\{m; |m-p_\alpha|=|p_\alpha|\}$ followed by the dilation $m\to |p_\alpha|^{-1}m$ 
\begin{equation}\label{galpha}
g_\alpha: m\mapsto \frac{p_\alpha}{|p_\alpha|}+|p_\alpha|\frac{(m-p_\alpha)}{|m-p_\alpha|^2} ,\quad p_\alpha= p^2_{\alpha}-p^1_{\alpha}\end{equation}
and this defines $h_\alpha:=g_\alpha\circ t_{\alpha}$. This map is an orientation reversing isometry of $\hh^{n+1}$ which maps $p^1_\alpha$ to $0$ and $p^2_\alpha$ to $\infty$, it conjugates $\alpha$ to a model form
\eqref{tgamma}.
 If $m=(x,y)\in \rr^+\x\rr^n$ are coordinates  on $\hh^{n+1}$, then 
 \begin{equation}\label{xtm}
 x(h_\alpha(m))=\frac{|p_\alpha| x}{|m-p^2_\alpha|^2}, \quad y(h_\alpha(m))=\frac{p_\alpha}{|p_\alpha|}+|p_\alpha|\frac{(y-p_\alpha^2)}{|m-p_\alpha^2|^2}.
 \end{equation} 
In particular, one has 
$ \max_{m\in \supp(a)} |h_\alpha(m)|\leq R$ for some $R$ depending only on $\supp(a)$. 
We notice that there exists $C>0$ uniform in $\gamma,\alpha$ such that for all $\gamma,\alpha$ 
\begin{equation}\label{miny}
\min \{|y|; m=(x,y)\in h_\alpha (\supp(a)), \alpha\in A_2(\gamma) \} \geq C.
\end{equation}
Indeed, by definition of $A_2(\gamma)$ and the fact that $h_\alpha$ is an isometry, 
the hyperbolic distance between $h_\alpha (\supp(a))$ and the line $y=0$ is bounded below by $1$, but 
$ C_0>x(h_\alpha(m))>C_1|p_\alpha|$ for some $C_0,C_1>0$ depending only 
on the support of $\supp(a)$. Then for $\eps>0$ small and $\alpha$ such that $|p_\alpha|\geq \eps>0$, there is a constant $C$ depending on $\eps$ such that \eqref{miny} holds for those $\alpha$. 
For the remaining $\alpha$ for which $|p_\alpha|\leq\eps$, it is easily seen from the expression \eqref{galpha}
that 
$|h_\alpha(m)|=1+O(\eps)$ for $m\in\supp(a)$ and therefore if $\eps$ is small enough $|y(h_\alpha (m))|>1/2$.

We now want to bound the terms in the right hand side of \eqref{intbyparts}. Using that $|m|$ 
is uniformly (with respect to $\alpha$) bounded on $h_\alpha (\supp(a))$ and using \eqref{dPhi}, \eqref{miny}, we deduce that there exists $C_N>0$ depending only on $N$ and $C>0$  such that for all $[\gamma]$, all $\alpha \in [\gamma]$, and  all $j\leq N$
\[\begin{gathered}
 \min_{m\in h_\alpha(\supp(a))} |(x\pl_x)\Phi_\alpha(m)|\geq C, \,\,\, 
  \sup_{m\in h_\alpha (\supp(a))} |Y_1\dots Y_j \Phi_\alpha(m)| \leq C_N,\\ 
  \sup_{m\in h_\alpha (\supp(a))} |Y_1\dots Y_j\sigma_\alpha(m)^\ndemi|\leq 
 C_N \sup_{\supp(a)} \sigma^\ndemi(d(m,\alpha m))\\
  \sup_{m\in h_\alpha (\supp(a))} |Y_1\dots Y_j\sigma_\alpha(m)| \leq C_N\sup_{\supp(a)} \sigma(d(m,\alpha m))\end{gathered}\] 
if $Y_i\in \{x\pl_x, x\pl_y\}$ (recall that $\sigma_\alpha=e^{\Phi_\alpha}$) Since $g_\alpha^{-1}(m)=|p_\alpha|g_\alpha(|p_\alpha|m)$  and
\[dg_\alpha(m)=|p_\alpha|\left( \frac{{\rm Id}}{|m-p_\alpha|^2}-2\frac{(m-p_\alpha)\cjg m-p_\alpha, \cdot \cjd}{|m-p_\alpha|^4} \right)\]
then we have 
\[ dg_\alpha^{-1}(g_\alpha (m))= |p_\alpha|^{-1}( |m-p_\alpha|^2{\rm Id}-2(m-p_\alpha)\cjg m-p_\alpha, \cdot \cjd )\]
and combining with  \eqref{xtm} and using $dh_\alpha(m)=dg_\alpha(m-p_\alpha^1)$
\begin{equation}\label{pullback} 
h_\alpha^*(xY)(m)= Y-2\frac{(m-p^2_\alpha)}{|m-p^2_\alpha|}\left\cjg \frac{m-p^2_\alpha}{|m-p^2_\alpha|^2}, Y \right\cjd   ,\quad Y\in \{\pl_x,\pl_y\}.\end{equation}
Combining with \eqref{integrale}, this shows that there is $C_N$ depending only on $N$ such that 
for all $\alpha\in A_2(\gamma)$
\[ \begin{gathered}
\Big|\int_{\supp(a)} a(m)F_{t}(\sigma(d(m,\alpha m))) dv(m)\Big| \\ 
\leq 
C_Nt^{-N+(n-1)/2} \int_{\supp(a)}\sigma(d(m,\alpha m))^{\ndemi}\max_{j\leq N} |(h_\alpha^* (x\pl_x))^N a(m)|dv(m).
\end{gathered}\]
and  by \eqref{pullback} we have 
$\sup_{m\in \supp(a)}|(h_\alpha^* (x\pl_x))^N a(m)|\leq C_N$ for some $C_N>0$ independent of $\alpha,\gamma$. 
We thus have proved the Lemma.
\end{proof}

We need to consider now the finitely many $\alpha\in A_1(\gamma)$ which contributes to the stationary phase (notice
that there are roughly $\mc{O}(\ell(\gamma))$ elements in $A_1(\gamma)$). Notice that, by the arguments above, the compact 
set  $h_\alpha(\supp(a))$ remains in a bounded region $\{C\geq |y|, C>|x|>1/C\}$ of $\hh^{n+1}$, uniformly in 
$\gamma,\alpha\in A_1(\gamma)$. This follows since the distance in the boundary between the fixed points of $\alpha$ 
is uniformly bounded below if the axis of $\alpha$ intersects $\supp(a)$. For the same reason, we obtain 
from the expression \eqref{galpha} that for  $\alpha\in A_1(\gamma)$   
\begin{equation}\label{deriveetalpha}
\max_{m\in h_\alpha (\supp(a))} |\pl_x^\mu\pl_y^\nu h_\alpha^{-1}(m)|\leq C
\end{equation}
for some $C$ depending on $\mu,\nu$ but  independent of $\gamma,\alpha$. 
The phase in the integral  
\[I_{k,\alpha}(t):=\int_{h_\alpha(\mc{F})}e^{-it r_\gamma(x,y)}b_k(\sigma_\alpha(x,y))
a(h^{-1}_\alpha(x,y))\frac{dxdy}{x^{n+1}}, \]
is stationary at $y=0$, and thus integrating on each slice $x={\rm cst}$, we 
have by stationary phase (\cite[Th. 7.7.5]{Hor})  
\begin{equation}\label{Ik}
\begin{gathered}
\left|I_{k,\alpha}(t)-e^{-it \ell(\gamma)}\sum_{j=0}^{N-1} t^{-\ndemi-j}\int_{\{y=0\}\cap h_\alpha(\mc{F})}
A_{2j}(x,y,\pl_y)[a(h_\alpha^{-1}(x,y))b_k(\sigma_\alpha(x,y))]\frac{dx}{x^{n+1}}\right|\\
\leq Ct^{-\ndemi-N}\sum_{|\beta|\leq 2N+2n}||\pl^\beta (b_k\circ\sigma_\alpha.
a\circ h^{-1}_\alpha)||_{L^\infty}.\end{gathered}
\end{equation}
for some smooth differential operator $A_{2j}(x,y,\pl_y)$ of order $2j$. By \cite[Th. 7.7.5]{Hor}, the constant $C$ is shown to be bounded 
uniformly with respect to $\alpha,\gamma$ if 
\[||r_\alpha||_{C^{3N+3n}(h_\alpha(\supp(a)))} \,\, \textrm{ and }\,\, \sup_{(x,y)\in h_\alpha(\supp(a))}\frac{|y|}{|dr_\alpha(x,y)|} \]
are bounded uniformly with respect to 
$\alpha \in A_1(\gamma)$ and with respect to $[\gamma]$. The differential operators $A_{2j}$ have coefficients 
bounded uniformly in $\alpha,\gamma$ if 
\[|\det({\rm Hess}(r_\alpha))|^{-1}\,\, \textrm{ and }  \,\, ||r_\alpha||_{C^{3N+3n}(h_\alpha(\supp(a)))}\]
are bounded uniformly. The uniform boundedness of semi-norms of $r_\alpha$ in a fixed 
compact set of $\hh^{n+1}$ (containing $h_\alpha(\supp(a))$)
follows directly from the expressions \eqref{drgamma} and \eqref{dPhi}, while determinant of Hessian
boundedness is clear from \eqref{hessien} since $\ell(\gamma)$ is bounded below by a positive constant (the radius of injectivity). By the explicit formula  \eqref{dPhi}, and using \eqref{drgamma}, it is direct that 
$|dr_\alpha(x,y)|\geq C|y|$ for some $C$ uniform in $\alpha$ and $(x,y)$ in a fixed compact set of $\hh^{n+1}$  
(using again that $\ell(\gamma)$ is bounded below by a uniform positive constant).  The constant $C$ in \eqref{Ik} is therefore uniform in $\alpha\in A_1(\gamma)$ and the conjugacy class $[\gamma]$.

Now we claim that the amplitudes $||\pl^\beta (b_k\circ\sigma_\alpha.
a\circ h^{-1}_\alpha)||_{L^\infty}\leq C \max_{m\in \supp(a)}e^{-\ndemi d(m,\alpha m)}$ with some 
$C$ uniform in $[\gamma],\alpha\in A_1(\gamma)$. By \eqref{deriveetalpha}, the terms involving $a\circ h_\alpha$
are clearly bounded uniformly. Moreover, since $b_k(\sigma)=\sigma^{\ndemi}c_k(\sigma)$, it 
suffices to bound the derivatives of $\Phi_\alpha=\log(\sigma_\alpha)$. By using that 
$h_\alpha\supp(a)$ is in a uniform compact set of $\hh^{n+1}$ with respect to $\alpha\in A_1(\gamma)$ and 
$[\gamma]$, we have $|\pl^\beta d\Phi|\leq C$ and thus 
\[ \sup_{m\in h_\alpha(\supp(a))}|\pl^\beta b_k(\sigma_\alpha(m))|\leq \sup_{m\in \supp(a)}e^{-\ndemi(d(m,\alpha m))}
.\]
Similarly all terms in the left hand side of \eqref{Ik} is bounded by 
$\sup_{m\in \supp(a)}e^{-\ndemi(d(m,\alpha m))}$. Notice that $\sigma_\alpha(m)=\cosh^{-1}(\ell(\gamma))$ 
is constant on $m\in \{y=0\}$, and thus we have proved 
\begin{lemma}\label{A_1}
With $I_{k,\alpha}$ defined in \eqref{Ik}, we have that for all $N\in \nn$, there exists $C_N>0$ 
such that for all $[\gamma]$, all $\alpha\in A_1(\gamma)$,  and all $t>1$
\[\begin{gathered}
\left|I_{k,\alpha}(t)-\frac{c_k((\cosh \ell(\gamma))^{-1}))}{(\cosh \ell(\gamma))^{\ndemi}}
e^{-it \ell(\gamma)}\sum_{j=0}^{N-1} t^{-\ndemi-j}\int_{{\rm axis}(\gamma)\cap 
\mc{F}}
Q_{2j}a(m) d\mu\right|
\\
\leq Ct^{-\ndemi-N}\sup_{m\in {\rm supp}(a)}e^{-\ndemi(d(m,\alpha m))}.\end{gathered}\]
for some differential operator $Q_{2j}$ with coefficients bounded uniformly in $\alpha,[\gamma]$, and
$d\mu$ is the measure induced by the riemannian measure on ${\rm axis}(\gamma)$.
\end{lemma}

In particular from \eqref{Ik} and the usual expression of $A_0(x,y,\pl_y)$  (see Th 3.14 in \cite{EvZw}) and \eqref{hessien}, we deduce
\[\begin{split}
I_{k,\alpha}(t)=& e^{-i(t+\ndemi) \ell(\gamma)}t^{-\ndemi}\frac{(2\pi)^{\ndemi}
e^{i\frac{n\pi}{4}}(\tanh \ell(\gamma))^{\ndemi} c_k(\cosh(\ell(\gamma))^{-1})}
{|\det({\rm Id}-e^{-\ell(\gamma)}R_\gamma^{-1})|}
\int_{{\rm axis}(\alpha)\cap \mc{F}}a(m)d\mu\\
& +\mc{O}(\sup_{m\in \supp(a)}e^{-\ndemi d(m,\alpha m)}t^{-n/2-1}).
\end{split}\]
where $d\mu$ is the measure induced by the riemannian measure on ${\rm axis}(\alpha)$.
From \eqref{c_0}, one also has 
\[c_0(\cosh(\ell(\gamma))^{-1}))=\sqrt{\pi}e^{-i\pi/4}(\tanh \ell(\gamma))^{-n/2}.\]

Summing Lemma \ref{A_1} and \ref{A_2} over all $\alpha\in [\gamma]$, we obtain 
the integrals over the closed geodesic associated to the conjugacy class $[\gamma]$, and then summing over the set of conjugacy classes $[\gamma]$, we obtain the result claimed in the Theorem. The sums over the group
converges since each estimate involving $\alpha\in[\gamma]$ 
has a bound in terms of  $\sup_{\supp(a)}e^{-\ndemi(d(m,\alpha m))}$, and the series 
\[\sum_{[\gamma]}\sum_{\alpha\in [\gamma]} e^{-\ndemi(d(m,\alpha m))}\]
converges uniformly for $m$ in compact subsets of $\hh^{n+1}$. This ends the proof.
\end{proof}

\end{document}